\documentclass[a4paper, 10pt, notitlepage]{article}

\usepackage{amsthm, amsmath, amssymb, latexsym}
\usepackage{mathrsfs,cite,eufrak}
\usepackage{graphicx}
\usepackage[ruled,vlined]{algorithm2e}

\theoremstyle{plain}
\newtheorem{theorem}{Theorem}
\newtheorem{lemma}{Lemma}
\newtheorem{corollary}{Corollary}

\theoremstyle{definition}
\newtheorem{definition}{Definition}
\newtheorem{example}{Example}

\theoremstyle{remark}
\newtheorem{remark}{Remark}

\DeclareMathOperator{\co}{co}
\DeclareMathOperator{\xL}{\textrm{L}}
\DeclareMathOperator{\dist}{dist}
\DeclareMathOperator*{\esssup}{ess\,sup}
\DeclareMathOperator{\domain}{dom}
\DeclareMathOperator{\graph}{Graph}

\author{M.V. Dolgopolik}
\title{Codifferentials and Quasidifferentials of the~Expectation of Nonsmooth Random Integrands and 
Two-Stage Stochastic Programming}

\begin{document}

\maketitle

\begin{abstract}
This work is devoted to an analysis of exact penalty functions and optimality conditions for nonsmooth two-stage
stochastic programming problems. To this end, we first study the co-/quasi-differentiability of the expectation of
nonsmooth random integrands and obtain explicit formulae for its co- and quasidifferential under some natural
assumptions on the integrand. Then we analyse exact penalty functions for a variational reformulation of
two-stage stochastic programming problems and obtain sufficient conditions for the global exactness of these functions
with two different penalty terms. In the end of the paper, we combine our results on the co-/quasi-differentiability of 
the expectation of nonsmooth random integrands and exact penalty functions to derive optimality conditions for nonsmooth
two-stage stochastic programming problems in terms of co\-dif\-fe\-rentials.
\end{abstract}

\section{Introduction}

Two-stage stochastic programming is one of the basic problems of stochastic optimization
\cite{ShapiroDentchevaRuszczynski,BirgeLouveaux} that has multiple applications in various fields, including
transportation planning \cite{BarbarosogluArda,LiuFanOrdonez}, disaster management \cite{GrassFischer}, optimal design
of energy systems \cite{ZhouZhangEtAl}, resources management \cite{HuangLoucks}, etc. Although two-stage stochastic
programming problems can be viewed as stochastic versions of bilevel optimization problems
\cite{DempeKalashnikov,DempeZemkoho}, their stochastic nature requires a largely different approach to their solution.
Optimality conditions for two-stage stochastic programming problems were obtained in 
\cite{RockafellarWets75,HiriartUrruty,XuYe10,ShapiroDentchevaRuszczynski,Vogel}, while
numerical methods for solving various classes of two-stage stochastic programming problems were studied e.g. in
\cite{ShapiroHomemDeMello,OliveiraSagastizabal,FabianSaoke,LeoveyRomisch} (see also the references therein).

The need for computing convex or nonconvex subdifferentials of the expectation of nonsmooth random integrands arises in
many areas of stochastic optimization, including two-stage stochastic programming, as well as stochastic linear
complementarity problems \cite{ChenFukushima}, stochastic variational inequalities \cite{ChenWetsZhang}, etc. The
subdifferential in the sense of convex analysis of the expectation of a convex integrand was computed in
\cite{RockafellarWets82}, while its approximations were discussed in \cite{NemirovskiJuditskyLanShapiro}. Various
approximations of the Clarke subdifferential of the expectation of nonsmooth random integrands were studied in
\cite{XuZhang09,BurkeChenSun}, while an outer estimate of its Mordukhovich basic subdifferential was obtained in
\cite{XuYe10}. Finally, a quasidifferential of the expectation of quasidifferentiable random integrands was computed in 
\cite{LinHuangXiaLi}.

The main goal of this paper is to apply constructive nonsmooth analysis
\cite{DemyanovDixon,DemyanovRubinov,DemyanovRubinov_collect} to a theoretical analysis of nonsmooth two-stage
stochastic programming problems. Firstly, we analyse the codifferentiability and quasidifferentiability of the
expectation of nonsmooth random integrands and present explicit formulae for its codifferential and quasidifferential in
the more general case and under different assumptions than in \cite{LinHuangXiaLi} (see
Remark~\ref{rmrk:DifferentQuasidiff} for more details).

In the second part of the paper we study exact penalty functions for two-stage stochastic programming problems,
reformulated as equivalent variational problems with pointwise constraints. With the use of the general theory of exact
penalty functions \cite{DiPilloGrippo,RubinovYang,Zaslavski,Dolgopolik_ExPenFunc,DolgopolikFominyh,Demyanov2010}, we
obtain sufficient conditions for the global exactness of penalty functions for two-stage stochastic programming with
two different types of penalty terms. The use of penalty terms of the first type leads to much less restrictive
assumptions on constraints of the second stage problem, while the second type of penalty terms is more convenient for
applications. In particular, it allows one to reformulate two-stage stochastic programming problems, whose second stage
problem has DC (Difference-of-Convex) objective function and DC constraints, as equivalent unconstrained DC optimization
problems and apply the well-developed apparatus of DC optimization to find their solutions (cf. analogous results for
bilevel programming problems in~\cite{StrekalovskyOrlov,Orlov}). Let us also note that exact penalty functions for
single-stage stochastic programming were analysed in \cite{FlamZowe}.

Finally, in the end of the paper we combine our results on quasidifferentials of the expectation of nonsmooth random
integrands and exact penalty functions for two-stage stochastic programming problems to obtains necessary optimality
conditions for these problems in terms of codifferentials.

The paper is organised as follows. Some auxiliary definitions and facts from constructive nonsmooth analysis, that are
necessary for understanding the paper, are collected in Section~\ref{sect:Preliminaries}. Codifferentiability and
quasidifferentiability of the expectation of nonsmooth random integrands is studied in
Section~\ref{sect:CodiffExpectation}, while Section~\ref{sect:TwoStageProgramming} is devoted to nonsmooth two-stage
stochastic programming problems. Exact penalty functions for such problems are analysed in
Subsection~\ref{subsect:ExactPenalty}, while optimality conditions for these problems in terms of codifferentials are
derived in Subsection~\ref{subsect:OptimalityConditions}.

\section{Preliminaries}
\label{sect:Preliminaries}

Let us introduce the notation and briefly recall several definitions from nonsmooth analysis that will be used
throughout the article. For more details in the finite dimensional case see
\cite{DemyanovDixon,DemyanovRubinov,DemyanovRubinov_collect}. The infinite dimensional case was studied in 
\cite{Dolgopolik_CodiffCalc,Dolgopolik_AbstractCodiff,Dolgopolik_MCD,DolgopolikMV_ESAIM}.

Let $X$ be a real Banach space. Denote by $X^*$ its topological dual, and by $\langle \cdot, \cdot \rangle$ the duality
pairing between $X$ and $X^*$. The weak${}^*$ topology on $X^*$ is denoted by $w^*$ or $\sigma(X^*, X)$ depending on
the context. Denote also by $\tau_{\mathbb{R}}$ the canonical topology of the real line $\mathbb{R}$. Let finally 
$U \subset X$ be an open set.

\begin{definition} \label{def:Codifferential}
A function $f \colon U \to \mathbb{R}$ is called \textit{codifferentiable} at a point $x \in U$, if there exists a pair
of convex subsets $\underline{d} f(x), \overline{d} f(x) \subset \mathbb{R} \times X^*$ that are compact in 
the topological product $(\mathbb{R} \times X^*, \tau_{\mathbb{R}} \times w^*)$, satisfy the equality
\begin{equation} \label{eq:CodiffZeroAtZero}
  \max_{(a, x^*) \in \underline{d} f(x)} a = \min_{(b, y^*) \in \overline{d} f(x)} b = 0,
\end{equation}
and for any $\Delta x \in X$ satisfy the following condition:
\begin{align*}
  \lim_{\alpha \to +0} \frac{1}{\alpha} \Big| f(x + \alpha \Delta x) - f(x) 
  &- \max_{(a, x^*) \in \underline{d} f(x)} \big( a + \langle x^*, \alpha \Delta x \rangle \big)
  \\
  &- \min_{(b, y^*) \in \overline{d} f(x)} \big( b + \langle y^*, \alpha \Delta x \rangle \big) \Big| = 0
\end{align*}
The pair $D f(x) = [\underline{d} f(x), \overline{d} f(x)]$ is called a \textit{codifferential} of $f$ at $x$, the set
$\underline{d} f(x)$ is referred to as a \textit{hypodifferential} of $f$ at $x$, while the set $\overline{d} f(x)$ is
called a \textit{hyperdifferential} of $f$ at $x$.
\end{definition}

\begin{remark}
{(i)~In the case when $X = \mathbb{R}^d$, a codifferential $D f(x)$ is a pair of convex compact subsets of
$\mathbb{R} \times \mathbb{R}^d = \mathbb{R}^{d + 1}$ satisfying the equalities from the previous definition. In
addition, if $X$ is a Hilbert space, then it is natural to suppose that a codifferential $D f(x)$ is a pair of convex
weakly compact subsets of the space $\mathbb{R} \times X$.
}

\noindent{(ii)~Note that a codifferential is not uniquely defined. In particular, one can easily verify that for any
compact convex subset $C$ of the space $(\mathbb{R} \times X^*, \tau_{\mathbb{R}} \times w^*)$ the pair 
$[ \underline{d} f(x) + C, \overline{d} f(x) - C ]$ is a codifferential of $f$ at $x$ as well.
}
\end{remark}

\begin{definition}
A function $f \colon U \to \mathbb{R}$ is called \textit{continuously} codifferentiable at a point $x \in U$, if $f$ is
codifferentiable at every point in a neighbourhood of $x$ and there exists a codifferential mapping 
$D f(\cdot) = [\underline{d} f(\cdot), \overline{d} f(\cdot)]$, defined in a neighbourhood of $x$ and such that the
multifunctions $\underline{d} f(\cdot)$ and $\overline{d} f(\cdot)$ are continuous in Hausdorff metric at $x$.
\end{definition}

The class of continuously codifferentiable at a given point (or on a given set) functions is closed under addition,
multiplication, composition with continuously differentiable functions, as well as pointwise maximum and minimum of
finite families of functions. Moreover, any convex function is continuously codifferentiable in a neighbourhood of any
given point from the interior of its effective domain, and any DC function (i.e. a function that can be represented as
the difference of convex functions) is continuously codifferentiable in a neighbourhood of any given point. Numerous
examples of continuously codifferentiable functions, as well as main rules of codifferential calculus can be found in
\cite{DemyanovDixon,DemyanovRubinov,DemyanovRubinov_collect,Dolgopolik_AbstractCodiff,Dolgopolik_MCD}.

\begin{definition} \label{def:Quasidifferential}
A function $f \colon U \to \mathbb{R}$ is called \textit{quasidifferentiable} at a point $x \in U$, if $f$ is
directionally differentiable at $x$ and its directional derivative $f'(x, \cdot)$ at this point can be represented as
the difference of sublinear functions or, equivalently, if there exists a pair 
$\underline{\partial} f(x), \overline{\partial} f(x) \subset X^*$ of compact weak${}^*$ compact sets such that
\[
  f'(x, h) = \max_{x^* \in \underline{\partial} f(x)} \langle x^*, h \rangle 
  + \min_{y^* \in \overline{\partial} f(x)} \langle y^*, h \rangle  \quad \forall h \in X.
\]
The pair $\mathscr{D} f(x) = [ \underline{\partial} f(x), \overline{\partial} f(x)]$ is called a
\textit{quasidifferential} of $f$ at $x$, the set $\underline{\partial} f(x)$ is called a \textit{subdifferential} of
$f$ at $x$, while the set $\overline{\partial} f(x)$ is referred to as a \textit{superdifferential} of $f$ at $x$.
\end{definition}

Just like codifferential, a quasidifferential is not uniquely defined. Here we only mention that a function $f$ is
codifferentiable at a point $x$ iff $f$ is quasidifferentiable at $x$ and one can easily compute a quasidifferential of
$f$ at $x$ from its codifferential at this point and vice versa. Namely, if $D f(x)$ is a codifferential of $f$ at $x$,
then the pair $\mathscr{D} f(x) = [\underline{\partial} f(x), \overline{\partial} f(x)]$ with
\begin{equation} \label{eq:QuasidiffViaCodiff}
  \underline{\partial} f(x) = \Big\{ x^* \in X^* \Bigm| (0, x^*) \in \underline{d} f(x) \Big\}, \quad
  \overline{\partial} f(x) = \Big\{ y^* \in X^* \Bigm| (0, y^*) \in \overline{d} f(x) \Big\}
\end{equation}
is a quasidifferential of $g$ at $x$. Conversely, if $\mathscr{D} f(x)$ is a quasidifferential of $f$ at $x$, then the
pair $[\{ 0 \} \times \underline{\partial} f(x), \{ 0 \} \times \overline{\partial} f(x)]$ is a codifferential of $f$
at $x$ (see, e.g. \cite{DemyanovRubinov,Dolgopolik_MCD}). Below we consider only quasidifferentials of the form 
\eqref{eq:QuasidiffViaCodiff}, that is, we suppose that if a codifferentiable function $f$ and its codifferential
$Df(x)$ are given, then $\mathscr{D} f(x)$ is a quasidifferential of $f$ of the form \eqref{eq:QuasidiffViaCodiff}.

Let us finally recall one auxiliary definition from set-valued analysis that will be used later (see, e.g.
\cite[Sect.~8.2]{AubinFrankowska} for more details). Let $X$ and $Y$ be metric spaces and $(\Omega, \mathfrak{A}, \mu)$
be a measure space. A set-valued mapping $F \colon X \times \Omega \rightrightarrows Y$, $F = F(x, \omega)$ is called
\textit{a Carath\'{e}odory map}, if for every $x \in X$ the multifunction $F(x, \cdot)$ is measurable and for a.e.
$\omega \in \Omega$ the multifunction $F(\cdot, \omega)$ is continuous.

\section{Codifferentials of the Expectation of Nonsmooth Random Integrands}
\label{sect:CodiffExpectation}

Let $(\Omega, \mathfrak{A}, P)$ be a probability space, and suppose that a nonsmooth function
$f \colon \mathbb{R}^d \times \mathbb{R}^m \times \Omega \to \mathbb{R}$, $f = f(x, y, \omega)$ is given.
In this section we study the codifferentiability of the nonsmooth integral functional
\[
  \mathcal{I}(x, y) = \mathbb{E}\big[ f(x, y(\cdot), \cdot) \big]
  := \int_{\Omega} f(x, y(\omega), \omega) \, d P(\omega),
\]
where $x \in \mathbb{R}^d$ is a parameter and $y \in \xL^p(\Omega, \mathfrak{A}, P; \mathbb{R}^m)$ with 
$1 < p \le + \infty$ is an $m$-dimensional random vector. Although the case $p = 1$ can be included into the general
theory under some additional assumptions, we exclude it for the sake of simplicity, since the proofs of the main results
below are much more cumbersome in the case $p = 1$, than in the case $1 < p \le + \infty$.

Denote by $p' \in [1, + \infty)$ the conjugate exponent of $p$, i.e. $1/p + 1/p' = 1$, and let $|\cdot|$ be 
the Euclidean norm in $\mathbb{R}^n$. Let us impose some assumptions on the integrand $f$ that, as we will show below,
ensure that the functional $\mathcal{I}$ is correctly defined and codifferentiable.

Namely, we will suppose that for a.e. $\omega \in \Omega$ and for all $(x, y) \in \mathbb{R}^d \times \mathbb{R}^m$
the function $f$ is codifferentiable jointly in $x$ and $y$, that is, there exists a pair of compact convex sets
$\underline{d}_{x, y} f(x, y, \omega), \overline{d}_{x, y} f(x, y, \omega) \subset 
\mathbb{R} \times \mathbb{R}^d \times \mathbb{R}^m$ such that
\[ 
  \Phi_f(x, y, \omega; 0, 0) = \Psi_f(x, y, \omega; 0, 0) = 0,
\]
and for all $(\Delta x, \Delta y) \in \mathbb{R}^d \times \mathbb{R}^m$ one has
\begin{multline*}
  \lim_{\alpha \to + 0} \frac{1}{\alpha} \Big| f(x + \alpha \Delta x, y + \alpha \Delta y, \omega) - f(x, y, \omega)
  \\
  - \Phi_f(x, y, \omega; \alpha \Delta x, \alpha \Delta y)
  - \Psi_f(x, y, \omega; \alpha \Delta x, \alpha \Delta y) \Big| = 0,
\end{multline*}
where
\begin{equation} \label{eq:CodiffPrimal}
\begin{split}
  \Phi_f(x, y, \omega; \Delta x, \Delta y) 
  &= \max_{(a, v_x, v_y) \in \underline{d}_{x, y} f(x, y, \omega)} 
  \big( a + \langle v_x, \Delta x \rangle + \langle v_y, \Delta y \rangle \big)
  \\
  \Psi_f(x, y, \omega; \Delta x, \Delta y)
  &= \min_{(b, w_x, w_y) \in \overline{d}_{x, y} f(x, y, \omega)}
  \big( b + \langle w_x, \Delta x \rangle + \langle w_y, \Delta y \rangle \big).
\end{split}
\end{equation}
The pair $D_{x, y} f(x, y, \omega) = [\underline{d}_{x, y} f(x, y, \omega), \overline{d}_{x, y} f(x, y, \omega)]$ is
called a codifferential of $f$ in $(x, y)$.

\vspace{3mm}

\noindent\textbf{Assumption~1.} The function $f$ satisfies the following conditions:
\begin{enumerate}
\item{for any $x \in \mathbb{R}^d$ the map $(y, \omega) \mapsto f(x, y, \omega)$ is a Carath\'{e}odory function;}

\item{the function $f$ satisfies the following growh condition of order $p$: for any $N > 0$ there exist
$C_N > 0$ and a nonnegative function $\beta_N \in \xL^1(\Omega, \mathfrak{A}, P)$ such that 
$|f(x, y, \omega)| \le \beta_N(\omega) + C_N |y|^p$ for all $x \in \mathbb{R}^d$ with $|x| \le N$, all 
$y \in \mathbb{R}^m$, and a.e. $\omega \in \Omega$ in the case $1 < p < + \infty$, and 
$|f(x, y, \omega)| \le \beta_N(\omega)$ for a.e. $\omega \in \Omega$ and all 
$(x, y) \in \mathbb{R}^d \times \mathbb{R}^m$ with $\max\{ |x|, |y| \} \le N$ in the case $p = + \infty$;
}

\item{the multifunctions $(y, \omega) \mapsto \underline{d}_{x, y} f(x, y, \omega)$ and
$(y, \omega) \mapsto \overline{d}_{x, y} f(x, y, \omega)$ are Carath\'{e}\-o\-dory maps for any $x \in \mathbb{R}^d$;
}

\item{the codifferential mapping $D_{x, y} f(\cdot)$ satisfies the following growth condition of order $p$: for any 
$N > 0$ there exist $C_N > 0$, and nonnegative functions $\beta_N \in \xL^1(\Omega, \mathfrak{A}, P)$ and
$\gamma_N \in \xL^{p'}(\Omega, \mathfrak{A}, P)$ such that
\[
  \max\{ |a|, |v_x| \} \le \beta_N(\omega) + C_N |y|^p, \quad
  |v_y| \le \gamma_N(\omega) + C_N |y|^{p - 1}
\]
for all $(a, v_x, v_y) \in \underline{d}_{x, y} f(x, y, \omega) \cup \overline{d}_{x, y} f(x, y, \omega)$, all 
$x \in \mathbb{R}^d$ with $|x| \le N$, all $y \in\mathbb{R}^m$, and a.e. $\omega \in \Omega$ in the case 
$1 < p < + \infty$, and 
\[
  \max\{ |a|, |v_x|, |v_y| \} \le \beta_N(\omega)
\]
for all $(a, v_x, v_y) \in \underline{d}_{x, y} f(x, y, \omega) \cup \overline{d}_{x, y} f(x, y, \omega)$, 
a.e. $\omega \in \Omega$, and for all vectors $(x, y) \in \mathbb{R}^d \times \mathbb{R}^m$ with 
$\max\{ |x|, |y| \} \le N$ in the case $p = + \infty$.
}
\end{enumerate}

Note that the Carath\'{e}odory and the growth conditions on the function $f$ ensure that the value $\mathcal{I}(x, y)$
is correctly defined and finite for all $x \in \mathbb{R}^d$ and $y \in \xL^p(\Omega, \mathfrak{A}, P; \mathbb{R}^m)$.
Let $X = \mathbb{R}^d \times \xL^p(\Omega, \mathfrak{A}, P; \mathbb{R}^m)$.

\begin{theorem} \label{thrm:Codifferentiability}
Let $1 < p \le + \infty$ and Assumption~1 be valid. Then the functional $\mathcal{I}$ is
codifferentiable on $\mathbb{R}^d \times \xL(\Omega, \mathfrak{A}, P; \mathbb{R}^m)$, and for any $(x, y)$ from this
space the pair $D \mathcal{I}(x, y) = [\underline{d} \mathcal{I}(x, y), \overline{d} \mathcal{I}(x, y)]$, defined as
\begin{multline} \label{def:HypodiffExpectValue}
  \underline{d} \mathcal{I}(x, y) = \Big\{ (A, x^*) \in \mathbb{R} \times X^* \Bigm|
  A = \mathbb{E}[a], 
  \\
  \langle x^*, (h_x, h_y) \rangle = \big\langle \mathbb{E}[v_x], h_x \big\rangle 
  + \int_{\Omega} \langle v_y(\omega), h_y(\omega) \rangle \, d P(\omega) \quad \forall (h_x, h_y) \in X, 
  \\
  (a(\cdot), v_x(\cdot), v_y(\cdot)) \text{ is a measurable selection of the map } 
  \underline{d}_{x, y} f(x, y(\cdot), \cdot) \Big\}
\end{multline}
and
\begin{multline*}
  \overline{d} \mathcal{I}(x, y) = \Big\{ (B, y^*) \in \mathbb{R} \times X^* \Bigm|
  B = \mathbb{E}[b], 
  \\
  \langle y^*, (h_x, h_y) \rangle = \big\langle \mathbb{E}[w_x], h_x \big\rangle 
  + \int_{\Omega} \langle w_y(\omega), h_y(\omega) \rangle \, d P(\omega) \quad \forall (h_x, h_y) \in X, 
  \\
  (b(\cdot), w_x(\cdot), w_y(\cdot)) \text{ is a measurable selection of the map } 
  \underline{d}_{x, y} f(x, y(\cdot), \cdot) \Big\},
\end{multline*}
is a codifferential of $\mathcal{I}$ at $(x, y)$.
\end{theorem}

The proof of Theorem~\ref{thrm:Codifferentiability} is similar to the proof of the codifferentiability of the
mapping $\mathcal{I}(u) = \int_{\Omega} f(x, u(x), \nabla u(x)) d x$ from the author's papers
\cite{DolgopolikMV_ESAIM,Dolgopolik_arXiv} (here $\Omega \subseteq \mathbb{R}^n$ is an open set and $u$ belongs to
the Sobolev space). On the other hand, Theorem~\ref{thrm:Codifferentiability} cannot be directly deduced from the main
results of \cite{DolgopolikMV_ESAIM,Dolgopolik_arXiv}. That is why below we present a detailed proof of
Theorem~\ref{thrm:Codifferentiability}. It seems possible to prove a more general result on the codifferentiability of
integral functionals defined on Banach spaces that subsumes Theorem~\ref{thrm:Codifferentiability} and the main results
of \cite{DolgopolikMV_ESAIM,Dolgopolik_arXiv} as particular cases. A development of such general theorem on the
codifferentiability of nonsmooth integral functionals is an interesting open problem for future research.

For the sake of convenience, we divide the proof of Theorem~\ref{thrm:Codifferentiability} into two lemmas.

\begin{lemma} \label{lem:CodifferentialSets}
Let $1 < p \le + \infty$ and Assumption~1 be valid. Then for any $(x, y) \in X$ the sets 
$\underline{d} \mathcal{I}(x, y)$ and $\overline{d} \mathcal{I}(x, y)$ from Theorem~\ref{thrm:Codifferentiability}
are nonempty, convex, compact in the topological product $(\mathbb{R} \times X^*, \tau_{\mathbb{R}} \times w^*)$, and
satisfy the following equalities:
\begin{equation} \label{eq:IntegralCodiffZeroAtZero}
  \max_{(A, x^*) \in \underline{d} \mathcal{I}(x, y)} A 
  = \min_{(B, y^*) \in \overline{d} \mathcal{I}(x, y)} B = 0.
\end{equation}
\end{lemma}

\begin{proof}
Fix any $(x, y) \in X$. We prove the statement of the lemma only for the hypodifferential 
$\underline{d} \mathcal{I}(x, y)$, since the proof for the hyperdifferential $\overline{d} \mathcal{I}(x, y)$ is exactly
the same.

By Assumption~1 the multifunction 
$(y, \omega) \mapsto \underline{d}_{x, y} f(x, y, \omega)$ is a Carath\'{e}odory map. Therefore by
\cite[Thrm.~8.2.8]{AubinFrankowska} the multifunction $\underline{d}_{x, y} f(x, y(\cdot), \cdot)$ is measurable, which
by \cite[Thrm.~8.1.3]{AubinFrankowska} implies that there exist a measurable selection 
$(a(\cdot), v_x(\cdot), v_y(\cdot))$ of this mapping. Furthermore, by the growth condition on the codifferential 
$D_{x, y} f(\cdot)$ from Assumption~1 all measurable selections of the set-valued mapping
$\underline{d}_{x, y} f(x, y(\cdot), \cdot)$ belong to the space 
\begin{equation} \label{eq:SpaceOfSelections}
  Y := \xL^1(\Omega, \mathfrak{A}, P) \times \xL^1(\Omega, \mathfrak{A}, P; \mathbb{R}^d)
  \times \xL^{p'}(\Omega, \mathfrak{A}, P; \mathbb{R}^m).
\end{equation}
Consequently, the linear functional $x^*$, defined as
\[
  \langle x^*, (h_x, h_y) \rangle = \big\langle \mathbb{E}[v_x], h_x \big\rangle 
  + \int_{\Omega} \langle v_y(\omega), h_y(\omega) \rangle d P(\omega) \quad \forall (h_x, h_y) \in X,
\]
belongs to $X^*$, and one can conclude that the hypodifferential $\underline{d} \mathcal{I}(x, y)$ is correctly defined
and nonempty.

Denote by $\mathcal{E}(x, y)$ the set of all measurable selections $z(\cdot) = (a(\cdot), v_x(\cdot), v_y(\cdot))$ of
the set-valued mapping $\underline{d}_{x, y} f(x, y(\cdot), \cdot)$. As was noted above, $\mathcal{E}(x, y)$ is a
subset of the space $Y$ defined in \eqref{eq:SpaceOfSelections}. For any $z = (a, v_x, v_y) \in Y$  denote by
$\mathcal{T}(z)$ the pair $(A, x^*)$ defined as in \eqref{def:HypodiffExpectValue}. Then 
$\underline{d} \mathcal{I}(x, y) = \mathcal{T}(\mathcal{E}(x, y))$.

By definition, for a.e. $\omega \in \Omega$ the hypodifferential $\underline{d}_{x, y} f(x, y(\omega), \omega)$ is a
convex set. Therefore the set of measurable selections $\mathcal{E}(x, y)$ of the multifunction 
$\underline{d}_{x, y} f(x, y(\cdot), \cdot)$ is convex. Hence taking into account the fact that the operator
$\mathcal{T}$ is linear one obtains that the hypodifferential $\underline{d} \mathcal{I}(x, y)$ is a convex set as the
image of a convex set under a linear map.

Recall that by the definition of hypodifferential one has $a \le 0$ for any 
$(a, v_x, v_y) \in \underline{d}_{x, y} f(x, y(\omega), \omega)$, $\omega \in \Omega$. Therefore $A \le 0$ for all 
$(A, x^*) \in \underline{d} \mathcal{I}(x, y)$. On the other hand, observe that thanks to equality
\eqref{eq:CodiffZeroAtZero} for a.e. $\omega \in \Omega$ one has
\[
  0 \in \Big\{ a \in \mathbb{R} \Bigm| \exists (v_x, v_y) \in \mathbb{R}^{d + m} \colon 
  (a, v_x, v_y) \in \underline{d}_{x, y} f(x, y(\omega), \omega) \Big\}.
\]
Hence by the Filippov theorem (see, e.g. \cite[Thrm.~8.2.10]{AubinFrankowska}) there exists a measurable selection
$(a_0(\cdot), v_{x0}(\cdot), v_{y0}(\cdot))$ of the set-valued map $\underline{d}_{x, y} f(x, y(\cdot), \cdot)$ such
that
$a_0(\omega) = 0$ almost surely. Consequently, for $(A_0, x^*_0) = \mathcal{T}(a_0, v_{x0}, v_{y0})$ one has $A_0 = 0$,
which implies that equality \eqref{eq:IntegralCodiffZeroAtZero} holds true.

Thus, it remains to prove the compactness of the set $\underline{d} \mathcal{I}(x, y)$ in the corresponding product
topology. To this end, let us verify that the set $\mathcal{E}(x, y)$ is a weakly compact subset of the space $Y$
defined in \eqref{eq:SpaceOfSelections}, and the operator $\mathcal{T}$ continuously maps the space $Y$ endowed with 
the weak topology to the topological product $(\mathbb{R}, \tau_{\mathbb{R}}) \times (X^*, w^*)$. Then one can conclude
that the hypodifferential $\underline{d} \mathcal{I}(x, y)$ is compact in the corresponding product topology as 
a continuous image of a compact set.

We start with the proof of the continuity of the operator $\mathcal{T}$. Let $\mathcal{V}$ be an open subset of the
product space $(\mathbb{R}, \tau_{\mathbb{R}}) \times (X^*, w^*)$. Let us show that its preimage 
$\mathcal{U} = \mathcal{T}^{-1}(\mathcal{V})$ under the map $\mathcal{T}$ is weakly open in $Y$. Indeed, fix any
$(a, v_x, v_y) \in \mathcal{U}$. Then $(A, x^*) = \mathcal{T}(a, v_x, v_y) \in \mathcal{V}$, which due to the openness
of the set $\mathcal{V}$ in the corresponding topology implies that there exist $\varepsilon > 0$, $n \in \mathbb{N}$,
and pairs $(h_i, \xi_i) \in X$, $i \in I = \{ 1, \ldots, n \}$,  such that
\[
  \mathcal{V}_{\varepsilon}(A, x^*) = \Big\{ (B, y^*) \in \mathbb{R} \times X^* \Bigm| \big| B - A \big| < \varepsilon,
  \quad \max_{i \in I} \big| \langle y^* - x^*, (h_i, \xi_i) \rangle \big| < \varepsilon \Big\} 
  \subseteq \mathcal{V}.
\]
Introduce the set
\begin{align*}
  \mathcal{U}_{\varepsilon}(a, v_x, v_y) = \bigg\{ (b, w_x, w_y) \in Y \bigg| 
  &\big| \mathbb{E}(b - a) \big| < \varepsilon, 
  \\
  &\max_{i \in I} \Big| \int_{\Omega} \langle w_x(\omega) - v_x(\omega), h_i \rangle \, d P(\omega) \Big| 
  < \frac{\varepsilon}{2},
  \\
  &\max_{i \in I} 
  \Big| \int_{\Omega} \langle w_y(\omega) - v_y(\omega), \xi_i(\omega) \rangle \, d P(\omega) \Big| 
  < \frac{\varepsilon}{2} \bigg\}.
\end{align*}
This set is neighbourhood of the point $(a, v_x, v_y)$ in the weak topology on $Y$. Moreover, by definition
$\mathcal{T}(\mathcal{U}_{\varepsilon}(a, v_x, v_y)) \subseteq \mathcal{V}_{\varepsilon}(A, x^*)$, which implies that
$\mathcal{U}_{\varepsilon}(a, v_x, v_y) \subseteq \mathcal{U}$. Thus, for any point $(a, v_x, v_y) \in \mathcal{U}$
there exists a neighbourhood of this point in the weak topology contained in $\mathcal{U}$. In other words, the set
$\mathcal{U}$ is weakly open, and one can conclude that the operator $\mathcal{T}$ is continuous with respect to the
chosen topologies.

Let us finally proof the weak compactness of the set $\mathcal{E}(x, y)$ in the space $Y$ defined in
\eqref{eq:SpaceOfSelections}. By the Eberlein-\v{S}mulian theorem it suffice to prove that $\mathcal{E}(x, y)$ is weakly
sequentially compact. To this end, choose any sequence 
$z_n(\cdot) = (a_n(\cdot), v_{xn}(\cdot), v_{yn}(\cdot)) \in \mathcal{E}(x, y)$, $n \in \mathbb{N}$. Let us consider
two cases.

\textbf{Case $p = + \infty$}. By the growth condition on the codifferential $D_{x, y} f(\cdot)$ (see
Assumption~1) there exists an a.e. nonnegative function 
$\beta \in \xL^1(\Omega, \mathfrak{A}, P)$ such that for a.e. $\omega \in \Omega$ one has
\[
  \max\big\{ |a_n(\omega)|, |v_{xn}(\omega)|, |v_{yn}(\omega)| \big\} \le \beta(\omega) 
  \quad \forall n \in \mathbb{N}.
\]
Hence by the weak compactness criterion in $\xL^1$ (see, e.g. \cite[Thrm.~4.7.20]{Bogachev}) the closures of the sets
$\{ a_n \}_{n \in \mathbb{N}}$, $\{ v_{xn} \}_{n \in \mathbb{N}}$, and $\{ v_{yn} \}_{n \in \mathbb{N}}$ are weakly
compact in the corresponding $\xL^1$ spaces. Therefore by the Eberlein-\v{S}mulian theorem there exists a subsequence
$z_{n_k} = (a_{n_k}, v_{xn_k}, v_{yn_k})$ weakly converging to some $z_*$ in $Y$. By Mazur's lemma there exists a
sequence of convex combinations $\{ \widehat{z}_k \}$ of elements of the sequence $z_{n_k}$ strongly converging to
$z_*$. Therefore, as is well known, there exists a subsequence $\{ \widehat{z}_{k_l} \}$ converging to $z_*$ almost
surely. 

Note that due to the convexity of $\mathcal{E}(x, y)$ one has $\{ \widehat{z}_k \} \subset \mathcal{E}(x, y)$, that is,
$\widehat{z}_k(\omega) \in \underline{d}_{x, y} f(x, y(\omega), \omega)$ for a.e. $\omega \in \Omega$ and all 
$k \in \mathbb{N}$. Hence taking into account the fact that by definition the hypodifferential 
$\underline{d}_{x, y} f(x, y(\omega), \omega)$, $\omega \in \Omega$, is a closed set, one obtains that 
$z_*(\omega) \in \underline{d}_{x, y} f(x, y(\omega), \omega)$ for a.e. $\omega \in \Omega$. Thus, 
$z_* \in \mathcal{E}(x, y)$, and the set $\mathcal{E}(x, y)$ is weakly sequentially compact, which completes the proof.

\textbf{Case $p < + \infty$}. By the growth condition on the codifferential $D_{x, y} f(\cdot)$ (see
Assumption~1) there exist $C > 0$ and a.e. nonnegative functions 
$\beta \in \xL^1(\Omega, \mathfrak{A}, P)$ and $\gamma \in \xL^{p'}(\Omega, \mathfrak{A}, P)$ such that for a.e. 
$\omega \in \Omega$ and all $n \in \mathbb{N}$ one has
\[
  \max\big\{ |a_n(\omega)|, |v_{xn}(\omega)| \big\} \le \beta(\omega) + C |y(\omega)|^p, \quad
  |v_{yn}(\omega)| \le \gamma(\omega) + C|y(\omega)|^{p - 1}.
\]
Observe that the right-hand side of the first inequality belongs to $\xL^1(\Omega, \mathfrak{A}, P)$, while the
right-hand side of the second one belongs to $\xL^{p'}(\Omega, \mathfrak{A}, P)$. Thus, the sequence
$\{ v_{yn} \}$ is norm-bounded in $\xL^{p'}(\Omega, \mathfrak{A}, P; \mathbb{R}^m)$, which due to the reflexivity of
this space (note that $1 < p' < + \infty$, since $1 < p < + \infty$) implies that there exists a weakly convergent
subsequence $\{ v_{y n_k} \}$. In turn, the existence of weakly convergence subsequences of the sequences $\{ a_n \}$
and $\{ v_{xn} \}$ follows from the weak compactness criterion in $\xL^1$ (see \cite[Thrm.~4.7.20]{Bogachev}). 

Thus, there exists a subsequence $\{ z_{n_k} \}$ weakly converging to some $z_* \in Y$. Now, applying Mazur's lemma and
arguing precisely in the same way as in the case $p = + \infty$ one can prove the weak compactness of the set 
$\mathcal{E}(x, y)$.
\end{proof}

Denote by $\| \cdot \|_p$ the standard norm on $\xL^p(\Omega, \mathfrak{A}, P)$.

\begin{lemma} \label{lem:AlmostCodifferentiability}
Let $1 < p \le + \infty$, Assumption~1 be valid, and the sets 
$\underline{d} \mathcal{I}(x, y)$ and $\overline{d} \mathcal{I}(x, y)$ be defined as in
Theorem~\ref{thrm:Codifferentiability}. Then for any 
$(x, y) \in X$ and $(\Delta x, \Delta y) \in X$ one has
\begin{multline*}
  \lim_{\alpha \to + 0} \frac{1}{\alpha} 
  \Big| \mathcal{I}(x + \alpha \Delta x, y + \alpha \Delta y) - \mathcal{I}(x, y) 
  - \max_{(A, x^*) \in \underline{d} \mathcal{I}(x, y)} 
  \big( A + \langle x^*, \alpha (\Delta x, \Delta y) \rangle \big)
  \\
  - \min_{(B, y^*) \in \overline{d} \mathcal{I}(x, y)} 
  \big( B + \langle y^*, \alpha (\Delta x, \Delta y) \rangle \big)
  \Big| = 0.
\end{multline*}
\end{lemma}

\begin{proof}
Fix any $(x, y) \in X$ and $(\Delta x, \Delta y) \in X$, and choose an arbitrary sequence 
$\{ \alpha_n \} \subset (0, + \infty)$ converging to zero. For a.e. $\omega \in \Omega$ and $n \in \mathbb{N}$ denote
\begin{multline} \label{eq:IntegrandIncrement}
  f_n(\omega) = \frac{1}{\alpha_n} 
  \Big( f(x + \alpha_n \Delta x, y(\omega) + \alpha_n \Delta y(\omega), \omega) - f(x, y(\omega), \omega)
  \\
  - \Phi_f\big( x, y(\omega), \omega; \alpha_n \Delta x, \alpha_n \Delta y(\omega) \big)
  - \Psi_f\big( x, y(\omega), \omega; \alpha_n \Delta x, \alpha_n \Delta y(\omega) \big) \Big),
\end{multline}
where the functions $\Phi_f$ and $\Psi_f$ are defined in \eqref{eq:CodiffPrimal}. By the definition of
codifferentiability the sequence $f_n$ converges to zero almost surely. Our aim is to prove that each term in the
definition of $f_n$ belongs to $\xL^1(\Omega, \mathfrak{A}, P)$ and there exists an a.e. nonnegative function 
$\rho \in \xL^1(\Omega, \mathfrak{A}, P)$ such that $|f_n| \le \rho$ almost surely. Then by Lebesgue's dominated
convergence theorem $\mathbb{E}[|f_n|] \to 0$ as $n \to \infty$. Hence integrating each term in the definition of
$f_n$ separately one obtains that
\begin{align*}
    \lim_{n \to \infty} \frac{1}{\alpha_n} 
  \Big| \mathcal{I}(x + \alpha_n \Delta x, &y + \alpha_n \Delta y) - \mathcal{I}(x, y) 
  \\
  &- \int_{\Omega} \Phi_f\big( x, y(\omega), \omega; \alpha_n \Delta x, \alpha_n \Delta y(\omega) \big) \, d P(\omega)
  \\
  &- \int_{\Omega} \Psi_f\big( x, y(\omega), \omega; \alpha_n \Delta x, \alpha_n \Delta y(\omega) \big) \, d P(\omega)
  \Big| = 0.
\end{align*}
Let us check that
\begin{multline} \label{eq:MaxOperIntechangable}
  \int_{\Omega} \Phi_f\big( x, y(\omega), \omega; \alpha_n \Delta x, \alpha_n \Delta y(\omega) \big) \, d P(\omega)
  \\
  = \max_{(A, x^*) \in \underline{d} \mathcal{I}(x, y)} 
  \big( A + \langle x^*, \alpha_n (\Delta x, \Delta y) \rangle \big)
\end{multline}
(a similar equality for the min terms involving the hyperdifferentials can be verified in the same way). Then one
obtains the desired result.

Indeed, by definition (see \eqref{eq:CodiffPrimal}) for any measurable selection $(a(\cdot), v_x(\cdot), v_y(\cdot))$ of
the set-valued mapping $\underline{d}_{x, y} f(x, y(\cdot), \cdot)$ one has
\[
  \Phi_f\big( x, y(\omega), \omega; \alpha_n \Delta x, \alpha_n \Delta y(\omega) \big)
  \ge a(\omega) + \langle v_x(\omega), \alpha_n \Delta x \rangle 
  + \langle v_y(\omega), \alpha_n \Delta y(\omega) \rangle,
\]
which implies that
\[
  \int_{\Omega} \Phi_f\big( x, y(\omega), \omega; \alpha_n \Delta x, \alpha_n \Delta y(\omega) \big) \, d P(\omega)
  \ge \max_{(A, x^*) \in \underline{d} \mathcal{I}(x, y)} \big( A + \langle x^*, \alpha_n \Delta x \rangle \big)
\]
(see \eqref{def:HypodiffExpectValue}). On the other hand, for a.e. $\omega \in \Omega$ one has
\begin{multline*}
  \Phi_f\big( x, y(\omega), \omega; \alpha_n \Delta x, \alpha_n \Delta y(\omega) \big)
  \\
  \in \Big\{ a + \langle v_x, \alpha_n \Delta x \rangle + \langle v_y, \alpha_n \Delta y(\omega) \rangle \Bigm| 
  (a, v_x, v_y) \in \underline{d}_{x, y} f(x, y(\omega), \omega) \Big\}.
\end{multline*}
Consequently, by the Filippov theorem (see, e.g. \cite[Thrm.~8.2.10]{AubinFrankowska}) there exists a measurable
selection $(a_0(\cdot), v_{x0}(\cdot), v_{y0}(\cdot))$ of the multifunction $\underline{d}_{x, y} f(x, y(\cdot), \cdot)$
such that
\[
  \Phi_f\big( x, y(\omega), \omega; \alpha_n \Delta x, \alpha_n \Delta y(\omega) \big)
  = a_0(\omega) + \langle v_{x0}(\omega), \alpha_n \Delta x \rangle 
  + \langle v_{y0}(\omega), \alpha_n \Delta y(\omega) \rangle
\]
for a.e. $\omega \in \Omega$. Hence for the corresponding pair 
$(A_0, x_0^*) = \mathcal{T}(a_0, v_{x0}, v_{y0})$ (see the proof of Lemma~\ref{lem:CodifferentialSets}), that by
definition belongs to $\underline{d} \mathcal{I}(x, y)$, one has
\[
  \int_{\Omega} \Phi_f\big( x, y(\omega), \omega; \alpha_n \Delta x, \alpha_n \Delta y(\omega) \big) \, d P(\omega)
  = A_0 + \langle x^*_0, \alpha_n (\Delta x, \Delta y) \rangle,
\]
and therefore equality \eqref{eq:MaxOperIntechangable} holds true.

Thus, it remains to show that Lebesgue's dominated convergence theorem is applicable to the sequence $\{ f_n \}$.
Indeed, the first two terms in the definition of $f_n$ (see~\eqref{eq:IntegrandIncrement}) belong to
$\xL^1(\Omega, \mathfrak{A}, P)$ by virtue of the first two parts of Assumption~1. Let us
check that these terms are dominated by a Lebesgue integrable function independent of $n$.

By the mean value theorem for codifferentiable functions \cite[Prp.~2]{Dolgopolik_MCD} for any $n \in \mathbb{N}$ and
for a.e. $\omega \in \Omega$ there exist $\alpha_n(\omega) \in (0, \alpha_n)$ and
\begin{align*}
  (0, v_{xn}(\omega), v_{yn}(\omega)) \in 
  \underline{d}_{x, y} f(x + \alpha_n(\omega) \Delta x, y(\omega) + \alpha_n(\omega) \Delta y(\omega), \omega), 
  \\
  (0, w_{xn}(\omega), w_{yn}(\omega)) \in 
  \overline{d}_{x, y} f(x + \alpha_n(\omega) \Delta x, y(\omega) + \alpha_n(\omega) \Delta y(\omega), \omega)
\end{align*}
such that
\begin{multline} \label{eq:MeanValueIntegrand}
  \frac{1}{\alpha_n} 
  \Big( f(x + \alpha_n \Delta x, y(\omega) + \alpha_n \Delta y(\omega), \omega) - f(x, y(\omega), \omega) \Big)
  \\
  = \langle v_{xn}(\omega) + w_{xn}(\omega), \Delta x \rangle 
  + \langle v_{yn}(\omega) + w_{yn}(\omega), \Delta y(\omega) \rangle.
\end{multline}
Put $\alpha_* = \max_{n \in \mathbb{N}} \alpha_n$. By the growth condition on the codifferential $D_{x, y} f$ (see
Assumption~1) there exist $C_N > 0$ and nonnegative functions 
$\beta_N \in \xL^1(\Omega, \mathfrak{A}, P)$ and $\gamma_N \in \xL^{p'}(\Omega, \mathfrak{A}, P)$ 
(here $N = |x| + \alpha_* |\Delta x|$) such that
\begin{align*}
  \max\big\{ |v_{xn}(\omega)|, |w_{xn}(\omega)| \big\} &\le
  \beta_N(\omega) + C_N \big| y(\omega) + \alpha_n(\omega) \Delta y(\omega) \big|^p
  \\
  &\le \beta_N(\omega) + C_N 2^p \Big( |y(\omega)|^p + \alpha_*^p |\Delta y(\omega)|^p \Big),
  \\
  \max\big\{ |v_{yn}(\omega)|, |w_{yn}(\omega)| \big\} &\le
  \gamma_N(\omega) + C_N \big| y(\omega) + \alpha_n(\omega) \Delta y(\omega) \big|^{p - 1}
\end{align*}
for a.e. $\omega \in \Omega$ and all $n \in \mathbb{N}$ in the case $1 < p < + \infty$, and there exists 
$\beta_N \in \xL^1(\Omega, \mathfrak{A}, P)$
(here $N = \max\{ |x| + \alpha_* |\Delta x|, \| y \|_{\infty} + \alpha_* \| \Delta y \|_{\infty} \}$) such that
\[
  \max\big\{ |v_{xn}(\omega)|, |w_{xn}(\omega)|, |v_{yn}(\omega)|, |w_{yn}(\omega)| \big\}
  \le \beta_N(\omega)
\]
for a.e. $\omega \in \Omega$  and all $n \in \mathbb{N}$ in the case $p = + \infty$. Hence with the use of
\eqref{eq:MeanValueIntegrand} one obtains that in the case $p = + \infty$ the inequality
\begin{multline*}
  \frac{1}{\alpha_n} 
  \Big| f(x + \alpha_n \Delta x, y(\omega) + \alpha_n \Delta y(\omega), \omega) - f(x, y(\omega), \omega) \Big|
  \\
  \le 2 \beta_N(\omega) |\Delta x| + 2 \beta_N(\omega) \| \Delta y \|_{\infty}
\end{multline*}
holds true for a.e. $\omega \in \Omega$ and all $n \in \mathbb{N}$, which implies that the first two terms in the
definition of $f_n$ (see~\eqref{eq:IntegrandIncrement}) are dominated by a Lebesgue integrable function independent of
$n$. In the case $p < + \infty$ one has
\begin{align*}
  \frac{1}{\alpha_n} 
  \Big| f(x + \alpha_n \Delta x, \: &y(\omega) + \alpha_n \Delta y(\omega), \omega) - f(x, y(\omega), \omega) \Big|
  \\
  &\le 2 \Big( \beta_N(\omega) + C_N 2^p \big( |y(\omega)|^p + \alpha_* |\Delta y(\omega)|^p \big) \Big) |\Delta x|
  \\
  &+ 2 
  \Big( \gamma_N(\omega) 
  + C_N 2^{p - 1} \big( |y(\omega)|^{p - 1} + \alpha_*^{p - 1} |\Delta y(\omega)|^{p - 1} \big) \Big) 
  |\Delta y(\omega)|
\end{align*}
The right-hand side of this inequality does not depend on $n$ and is Lebesgue integrable, as one can easily verify with
the use of H\"{o}lder's inequality and the equality $p'(p - 1) = p$. Thus, in the case $p < + \infty$ the first two
terms in the definition of $f_n$ are dominated by a Lebesgue integrable function independent of $n$ as well.

Let us finally check that the third term in the definition of $f_n$, denoted by
\[
  \theta_n(\omega) := \frac{1}{\alpha_n} \max_{(a, v_x, v_y) \in \underline{d}_{x, y} f(x, y(\omega), \omega)} 
  \big( a + \langle v_x, \alpha_n \Delta x \rangle + \langle v_y, \alpha_n \Delta y(\omega) \rangle \big)
\]
(see~\eqref{eq:IntegrandIncrement}), is measurable and dominated by a Lebesgue integrable function independent of $n$.
The fact that the last term (the min term) in the definition of $f_n$ is measurable and dominated by a Lebesgue
integrable function independent of $n$ is proved in exactly the same way.

As was shown in the proof of Lemma~\ref{lem:CodifferentialSets}, the set-valued mapping 
$\underline{d}_{x, y} f(x, y(\cdot), \cdot)$ is measurable. Consequently, the function $\theta_n$ is measurable by
\cite[Thrm.~8.2.11]{AubinFrankowska}. 

For any $\omega \in \Omega$ introduce the function
\[
  g_{\omega}(t) = \max_{(a, v_x, v_y) \in \underline{d}_{x, y} f(x, y(\omega), \omega)} 
  \big( a + \langle v_x, t \Delta x \rangle + \langle v_y, t \Delta y(\omega) \rangle \big).
\]
Observe that by the definition of codifferential $g_{\omega}(0) = 0$ (see Def.~\ref{def:Codifferential}) and for any 
$t, \Delta t \in \mathbb{R}$ and $\alpha > 0$ one has
\[
  \frac{1}{\alpha} \Big| g_{\omega}(t + \alpha \Delta t) - g_{\omega}(t) 
  - \max_{(a_g, v_g) \in \underline{d} g_{\omega}(t)} \big( a_g + v_g (\alpha \Delta t) \big) \Big| = 0,
\]
where
\begin{multline*}
  \underline{d} g_{\omega}(t) = \Big\{ (a_g, v_g) \in \mathbb{R} \times \mathbb{R} \Bigm|
  a_g = a + \langle v_x, t \Delta x \rangle + \langle v_y, t \Delta y(\omega) \rangle - g_{\omega}(t), 
  \\
  v_g = \langle v_x, \Delta x \rangle + \langle v_y, \Delta y(\omega) \rangle, \:
  (a, v_x, v_y) \in \underline{d}_{x, y} f(x, y(\omega), \omega) \Big\}.
\end{multline*}
The set $\underline{d} g_{\omega}(t)$ is obviously convex and compact. Moreover, note that the equality
$\max\{ a_g \mid (a_g, v_g) \in \underline{d} g_{\omega}(t) \} = g_{\omega}(t) - g_{\omega}(t) = 0$ holds true. Thus,
the function $g_{\omega}$ is codifferentiable at every point $t \in \mathbb{R}$, and the pair $[\underline{d}
g_{\omega}(t), \{ 0 \}]$ is a codifferential of $g_{\omega}$ at the point $t$. 

Applying the mean value theorem for codifferentiable functions \cite[Prp.~2]{Dolgopolik_MCD} one obtains that
for any $n \in \mathbb{N}$ and for a.e. $\omega \in \Omega$ there exists $\alpha_n(\omega) \in (0, \alpha_n)$ and
$(0, v_{gn}(\omega)) \in \underline{d} g_{\omega}(\alpha_n(\omega))$ such that 
\[
  \theta_n(\omega) = \frac{1}{\alpha_n} (g_{\omega}(\alpha_n) - g_{\omega}(0)) = v_g(\omega)
\]
or, equivalently, there exists 
$(a_n(\omega), v_{xn}(\omega), v_{yn}(\omega)) \in \underline{d}_{x, y} f(x, y(\omega), \omega)$ such that
\[
  \theta_n(\omega) = \langle v_{xn}(\omega), \Delta x \rangle + \langle v_{yn}(\omega), \Delta y(\omega) \rangle
  \quad \forall n \in \mathbb{N}
\]
Hence by the growth condition on the codifferential $D_{x, y} f$ (see Assumption~1)
there exist $C_N > 0$ and a.e. nonnegative functions $\beta_N \in \xL^1(\Omega, \mathfrak{A}, P)$ and
$\gamma_N \in \xL^{p'}(\Omega, \mathfrak{A}, P)$ (here $N = \| x \|$) satisfying the inequality
\[
  |\theta_n(\omega)| \le \Big( \beta_N(\omega) + C_N |y(\omega)|^p \Big) |\Delta x|
  + \Big( \gamma_N(\omega) + C_N |y(\omega)|^{p - 1} \Big) |\Delta y(\omega)|
\]
for a.e. $\omega \in \Omega$ in the case $p < + \infty$, and the inequality
\[
  |\theta_n(\omega)| \le \beta_N(\omega) |\Delta x| + \beta_N(\omega) \| \Delta y(\omega) \|_{\infty}
\]
for a.e. $\omega \in \Omega$ in the case $p = + \infty$. The right-hand sides of these inequalities are Lebesgue
integrable and do not depend on $n$. Thus, the sequence $\{ \theta_n \}$ is dominated by a Lebesgue integrable function,
which completes the proof.
\end{proof}

With the use of Theorem~\ref{thrm:Codifferentiability} one can easily obtain sufficient conditions for 
the quasidifferentiability of the functional $\mathcal{I}$. Recall that
$X = \mathbb{R}^d \times \xL^p(\Omega, \mathfrak{A}, P; \mathbb{R}^m)$.

\begin{corollary} \label{crlr:Quasidifferentiability}
Let $1 < p \le + \infty$ and Assumption~1 be valid. Then the functional $\mathcal{I}$ is
quasidifferentiable on $\mathbb{R}^d \times \xL(\Omega, \mathfrak{A}, P; \mathbb{R}^m)$, and for any $(x, y)$ from this
space the pair 
$\mathscr{D} \mathcal{I}(x, y) = [\underline{\partial} \mathcal{I}(x, y), \overline{\partial} \mathcal{I}(x, y)]$,
defined as
\begin{multline*}
  \underline{\partial} \mathcal{I}(x, y) = \Big\{ x^* \in X^* \Bigm| 
  \langle x^*, (h_x, h_y) \rangle = \big\langle \mathbb{E}[v_x], h_x \big\rangle 
  + \int_{\Omega} \langle v_y(\omega), h_y(\omega) \rangle d P(\omega) 
  \\ 
  \forall (h_x, h_y) \in X, \quad
  (0, v_x(\cdot), v_y(\cdot)) \text{ is a measurable selection of } 
  \underline{d}_{x, y} f(x, y(\cdot), \cdot) \Big\}
\end{multline*}
and
\begin{multline*}
  \overline{\partial} \mathcal{I}(x, y) = \Big\{ y^* \in X^* \Bigm| 
  \langle y^*, (h_x, h_y) \rangle = \big\langle \mathbb{E}[w_x], h_x \big\rangle 
  + \int_{\Omega} \langle w_y(\omega), h_y(\omega) \rangle d P(\omega) 
  \\ 
  \forall (h_x, h_y) \in X, \quad
  (0, w_x(\cdot), w_y(\cdot)) \text{ is a measurable selection of } 
  \underline{d}_{x, y} f(x, y(\cdot), \cdot) \Big\},
\end{multline*}
is a quasidifferential of $\mathcal{I}$ at $(x, y)$. Moreover, the following equality holds true:
\begin{equation} \label{eq:DirectDerivIntegration}
  \mathcal{I}'(x, y; h_x, h_y) 
  = \int_{\Omega} \big[ f_0(\cdot, \cdot, \omega) \big]'(x, y(\omega); h_x, h_y(\omega)) \, d P(\omega)
  \quad \forall (h_x, h_y) \in X.
\end{equation}
\end{corollary}

\begin{proof}
Applying Theorem~\ref{thrm:Codifferentiability} and the fact that any codifferentiable function $g$ with codifferential
$D g(x)$ is quasidifferentiable and the pair
\[
  \underline{\partial} g(x) = \Big\{ x^* \in X^* \Bigm| (0, x^*) \in \underline{d} g(x) \Big\}, \quad
  \overline{\partial} g(x) = \Big\{ y^* \in X^* \Bigm| (0, y^*) \in \overline{d} g(x) \Big\}
\]
is a quasidifferential of $g$ at $x$ (see, e.g. \cite{DemyanovRubinov,Dolgopolik_MCD}), one obtains the required results
on the quasidifferentiability of the functional $\mathcal{I}$. 

To prove equality \eqref{eq:DirectDerivIntegration}, recall that the set-valued mappings 
$\underline{d}_{x, y} f(x, y(\cdot), \cdot)$ and $\overline{d}_{x, y} f(x, y(\cdot), \cdot)$ are measurable, as was
shown in the proof of Lemma~\ref{lem:CodifferentialSets}. Hence with the use of \cite[Thrm.~8.2.4]{AubinFrankowska} one
obtains that the set-valued mappings $\underline{\partial}_{x, y} f(x, y(\cdot), \cdot)$ and 
$\overline{\partial}_{x, y} f(x, y(\cdot), \cdot)$, defined according to equalities \eqref{eq:QuasidiffViaCodiff}, are
measurable as well. Consequently, applying the definition of quasidifferentiability and arguing in the
same way as in the proof of Lemma~\ref{lem:AlmostCodifferentiability} (or utilising the interchangeability principle;
see, e.g. \cite[Thrm.~14.60]{RockafellarWets}) one gets that
\begin{multline*}
  \mathcal{I}'(x, y; h_x, h_y) 
  = \int_{\Omega} \Big( \max_{(v_x, v_y) \in \underline{\partial}_{x, y} f(x_*, y_*(\omega), \omega)} 
  \big( \langle v_x, h_x \rangle + \langle v_y, h_y(\omega) \rangle \big)
  \\
  + \min_{(w_x, w_y) \in \overline{\partial}_{x, y} f(x_*, y_*(\omega), \omega)} 
  \big( \langle w_x, h_x \rangle + \langle w_y, h_y(\omega) \rangle \big) \Big) \, d P(\omega)
\end{multline*}
for all $(h_x, h_y) \in X$, which by the definition of quasidifferential of the function $f$ implies that equality
\eqref{eq:DirectDerivIntegration} holds true.
\end{proof}

\begin{remark} \label{rmrk:DifferentQuasidiff}
In the particular case when the function $f$ does not depent on $y$, i.e. $f = f(x, \omega)$, the previous corollary
contains sufficient conditions for the quasidifferentiability of the function $F(x) = \mathbb{E}[f(x, \cdot)]$.
Quasidifferentiability of this function was studied in the recent paper \cite{LinHuangXiaLi} under different assumptions
on the function $f$. Namely, instead of imposing any growth conditions, in \cite{LinHuangXiaLi} it was assumed that all
integrals are correctly defined and the function $f$ is locally Lipschitz continuous in $x$ uniformly in $\omega$. 
\end{remark}

Let us finally show that under the assumptions of Theorem~\ref{thrm:Codifferentiability} the functional 
$\mathcal{I}(x, y)$ is not only codifferentiable, but also Lipschitz continuous on bounded sets.

\begin{corollary} \label{crlr:LipschitzCont}
Let $1 < p \le + \infty$ and Assumption~1 be valid. Then $\mathcal{I}$ is
Lipschitz continuous on any bounded subset of the space 
$X = \mathbb{R}^d \times \xL^p(\Omega, \mathfrak{A}, P; \mathbb{R}^m)$.
\end{corollary}

\begin{proof}
With the use of the growth condition on the codifferential mapping $D_{x, y} f(\cdot)$ from
Assumption~1 one can readily verify that both multifunctions
$\underline{d} \mathcal{I}(\cdot)$ and $\overline{d} \mathcal{I}(\cdot)$ are bounded on bounded subsets of the space
$X$. Therefore by  \cite[Corollary~2]{Dolgopolik_MCD} the functional $\mathcal{I}$ is Lipschitz continuous on any
bounded subset of this space.
\end{proof}

\section{Nonsmooth Two-Stage Stochastic Programming}
\label{sect:TwoStageProgramming}

Let, as above, $(\Omega, \mathfrak{A}, P)$ be a probability space. In this section we study a general two-stage
stochastic programming problem of the form
\begin{equation} \label{prob:FirstStage}
  \min_{x \in A} \mathbb{E}\big[ F(x, \omega) \big],
\end{equation}
where $F(x, \omega)$ is the optimal value of the second stage problem
\begin{equation} \label{prob:SecondStage}
  \min_{y \in G(x, \omega)} \: f(x, y, \omega).
\end{equation}
Here $A \subset \mathbb{R}^d$ is a closed set, $f \colon \mathbb{R}^d \times \mathbb{R}^m \times \Omega \to \mathbb{R}$
is a Carath\'{e}odory function, and $G \colon \mathbb{R}^d \times \Omega \rightrightarrows \mathbb{R}^m$ is a
multifunction. We assume that $G$ is measurable and for every $\omega \in \Omega$ the multifunction $G(\cdot, \omega)$
is closed.

Choose any $1 \le p \le + \infty$, and denote $X = \mathbb{R}^d \times \xL^p(\Omega, \mathfrak{A}, P; \mathbb{R}^m)$. By
the interchangeability principle for two-stage stochastic programming (see, e.g. 
\cite[Thrm.~2.20]{ShapiroDentchevaRuszczynski}), problem \eqref{prob:FirstStage}, \eqref{prob:SecondStage} is
equivalent the following variational problem with pointwise constraints:
\begin{equation*}
\begin{split}
  &\min_{(x, y) \in X} \mathbb{E}\big[ f(x, y(\cdot), \cdot) \big] 
  \\
  &\text{subject to} \quad x \in A, \quad y(\omega) \in G(x, \omega) \quad \text{for a.e.} \quad \omega \in \Omega,
\end{split}
  \qquad \qquad (\mathcal{P})
\end{equation*}
in the sense that the optimal values of these problems coincide, and if this common optimal value is finite, then for
any globally optimal solution $(x_*, y_*(\cdot))$ of the problem $(\mathcal{P})$ the point $x_*$ is a globally optimal
solution of problem \eqref{prob:FirstStage} and for a.e. $\omega \in \Omega$ the point $y_*(\omega)$ is a globally
optimal solution of the second stage problem \eqref{prob:SecondStage}. Conversely, if $x_*$ is a globally optimal
solution of problem \eqref{prob:FirstStage} and for a.e. $\omega \in \Omega$ the point $y_*(\omega)$ is a globally
optimal solution of problem \eqref{prob:SecondStage} with $x = x_*$ such that 
$y_* \in \xL^p(\Omega, \mathfrak{A}, P; \mathbb{R}^m)$, then $(x_*, y_*)$ is a globally optimal solution of the problem
$(\mathcal{P})$.

Since problem \eqref{prob:FirstStage}, \eqref{prob:SecondStage} and the problem $(\mathcal{P})$ are equivalent, below
we consider only the problem $(\mathcal{P})$. Our aim is to present several results on exact penalty functions for 
the problem $(\mathcal{P})$, which not only allow one to obtain optimality conditions for the original two-stage
stochastic programming problem, but also can be used for design and analysis of exact penalty methods for solving 
problem \eqref{prob:FirstStage}, \eqref{prob:SecondStage}.

\subsection{Exact Penalty Functions}
\label{subsect:ExactPenalty}

Fix any $p \in [1, + \infty]$, and denote by
\[
  \mathcal{I}(x, y) = \int_{\Omega} f(x, y(\omega), \omega) \, d P(\omega)
\]
the objective function of the problem $(\mathcal{P})$. Below we suppose that the functional $\mathcal{I}$ is correctly
defined on the space $X := \mathbb{R}^d \times \xL^p(\Omega, \mathfrak{A}, P; \mathbb{R}^m)$ and does not take the
value 
$- \infty$. In particular, it is sufficient to suppose that for any $x \in \mathbb{R}^d$ there exist $C > 0$ and an
a.e. nonnegative function $\beta \in \xL^1(\Omega, \mathfrak{A}, P)$ such that
$|f(x, y, \omega)| \le \beta(\omega) + C |y|^p$ for a.e. $\omega \in \Omega$ and all $y \in \mathbb{R}^m$ in the case 
$p < + \infty$, and for any $x \in \mathbb{R}^d$ and $N > 0$ there exist an a.e. nonnegative function 
$\beta_N \in \xL^1(\Omega, \mathfrak{A}, P)$ such that $|f(x, y, \omega)| \le \beta_N(\omega)$ for a.e. 
$\omega \in \Omega$ and all $y \in \mathbb{R}^m$ with $|y| \le N$.

Introduce the set
\[
  M = \Big\{ (x, y) \in X \Bigm| y(\omega) \in G(x, \omega) \text{ for a.e. } \omega \in \Omega \Big\}.
\]
Then the problem $(\mathcal{P})$ can be rewritten as follows:
\[
  \min_{(x, y) \in X} \: \mathcal{I}(x, y) \quad \text{subject to} \quad
  (x, y) \in M \cap (A \times \xL^p(\Omega, \mathfrak{A}, P; \mathbb{R}^m)).
\]
Let $\varphi \colon X \to [0, + \infty]$ be any function such that $\varphi(x, y) = 0$ iff $(x, y) \in M$, and let
$\Phi_c(x, y) = \mathcal{I}(x, y) + c \varphi(x, y)$. The function $\Phi_c$ is called a \textit{penalty function} for
the problem $(\mathcal{P})$ with $c \ge 0$ being the penalty parameter, while the function $\varphi$ is called
\textit{a penalty term} for the constrain $(x, y) \in M$. Our aim is to obtain sufficient conditions for the
\textit{exactness} of the penalty function $\Phi_c$.

Recall that the penalty function $\Phi_c$ is called \textit{ globally exact}, if there exists $c_* \ge 0$ such that for
any $c \ge c_*$ the set of globally optimal solutions of the penalized problem
\begin{equation} \label{prob:PenalizedProblem}
  \min_{(x, y) \in X} \: \Phi_c(x, y) \quad \text{subject to} \quad x \in A
\end{equation}
coincides with the set of globally optimal solutions of the problem $(\mathcal{P})$. The greatest lower bound of all
such $c_*$ is called \textit{the least exact penalty parameter} of the penalty function $\Phi_c$. One can verify that
the penalty function $\Phi_c$ is globally exact iff there exists $c_* \ge 0$ such that for any $c \ge c_*$ the problem
$(\mathcal{P})$ and problem \eqref{prob:PenalizedProblem} have the same optimal value, and the greatest lower bound of
all such $c_*$ coincides with the least exact penalty parameter. See
\cite{DiPilloGrippo,RubinovYang,Zaslavski,Dolgopolik_ExPenFunc,DolgopolikFominyh,Demyanov2010} for more details on exact
penalty functions.

Let us obtain sufficient conditions for the global exactness of the penalty function $\Phi_c$ with the penalty term
$\varphi$ defined in several different ways. To this end, we will utilise general sufficient conditions for the
exactness of penalty functions in metric and normed spaces from \cite{Dolgopolik_ExPenFunc,DolgopolikFominyh}, and the
following auxiliary lemma, which is a slight generalization of \cite[Prp.~3.13]{Dolgopolik_ExPenFunc}. 

\begin{lemma} \label{lem:LipschitzFuncLowerEstimate}
Let $Y$ be a normed space, $\mathcal{F} \subset Y$ be a nonempty set, and a function
$F \colon Y \to \mathbb{R} \cup \{ + \infty \}$ be such that for any bounded set $C \subset Y$ there exists 
a continuous from the right function $\omega_C \colon [0, + \infty) \to [0, + \infty)$ for which
\begin{equation} \label{eq:ContinuityModulus}
  \big| F(y_1) - F(y_2) \big| \le \omega_C \big( \| y_1 - y_2 \| \big) \quad \forall y_1, y_2 \in C.
\end{equation}
Then for any $R > 0$ there exists a bounded set $C \subset Y$ such that
\begin{equation} \label{eq:LowerEstimateViaContMod}
  F(y) \ge \inf_{z \in \mathcal{F}} F(z) - \omega_C \big( \dist(y, \mathcal{F}) \big)
  \quad \forall y \in B(0, R) = \{ z \in Y \mid \| z \| \le R \}.
\end{equation}
\end{lemma}

\begin{proof}
Denote $F_* = \inf_{z \in \mathcal{F}} F(z)$, and fix any $R > 0$ and $z \in \mathcal{F}$. By our assumption
there exists a continuous from the right function $\omega_C$ such that inequality \eqref{eq:ContinuityModulus} holds
true for $C = B(0, R + \| z \|)$.

Choose any $y \in B(0, R)$. If $y \in \mathcal{F}$, then inequality \eqref{eq:LowerEstimateViaContMod} trivially holds
true. Suppose now that $y \in B(0, R) \setminus \mathcal{F}$. Clearly, there exists a sequence 
$\{ y_n \} \subset \mathcal{F}$ such that $\| y - y_n \| \to \dist(y, \mathcal{F})$ as 
$n \to \infty$, and the inequalities $\| y - y_n \| \le \| y - z \| \le R + \| z \|$ and 
$\| y - y_n \| \ge \| y - y_{n + 1} \|$ are satisfied for all $n \in \mathbb{N}$. By definition 
$\{ y_n \} \subset C$, $y \in C$, and $F(y_n) \ge F_*$ for all $n \in \mathbb{N}$. Therefore, by applying inequality
\eqref{eq:ContinuityModulus} one obtains that
\[
  F_* - F(y) = F_* - F(y_n) + F(y_n) - F(y) \le F(y_n) - F(y) \le \omega_C\big( \| y - y_n \| \big)
\]
for any $n \in \mathbb{N}$. Hence passing to the limit as $n \to \infty$ with the use of the fact that the function
$\omega_C$ is continuous from the right and the sequence $\{ \| y - y_n \| \}$ is non-increasing one gets that
inequality \eqref{eq:LowerEstimateViaContMod} holds true.
\end{proof}

\begin{remark}
Note that if $F$ is Lipschitz continuous on bounded sets, then inequality \eqref{eq:ContinuityModulus} holds true with
$\omega_C(t) = L_C t$, where $L_C$ is a Lipschitz constant of $F$ on $C$. In this case the statement of the lemma can be
reformulated as follows: for any $R > 0$ there exists $L > 0$ such that $F(y) \ge F_* - L \dist(y, \mathcal{F})$ for all
$y \in B(0, R)$. Thus, Lemma~\ref{lem:LipschitzFuncLowerEstimate} provides a lower estimate of the decay of the function
$F$ relative to a given set $\mathcal{F}$.
\end{remark}

We start our analysis of the exactness of the penalty function $\Phi_c$ with the simplest case when the penalty term
$\varphi$ is defined via the distance function to the multifunction $G$. Denote by $\mathcal{I}_*$ the optimal value of
the problem $(\mathcal{P})$.

\begin{theorem} \label{thrm:ExactPenaltyDistToSet}
Let there exist a globally optimal solution of the problem $(\mathcal{P})$, the set-valued mapping $G$ have closed
images, and
\[
  \varphi(x, y) = \Big( \mathbb{E}[\dist(y(\cdot), G(x, \cdot))^p] \Big)^{1/p} \quad \forall (x, y) \in X
\]
in the case $p < + \infty$, and $\varphi(x, y) = \esssup_{\omega \in \Omega} \dist(y(\omega), G(x, \omega))$ for all 
$(x, y) \in X$ in the case $p = + \infty$. Suppose also that the functional $\mathcal{I}$ is Lipschitz continuous on
bounded sets, and there exists $c \ge 0$ such that the set 
\[
  \{ (x, y) \in X \mid x \in A, \: \Phi_c(x, y) < \mathcal{I}_* \}
\]
is bounded. Then the penalty function $\Phi_c$ is globally exact.
\end{theorem}

\begin{proof}
Observe that the function $\varphi$ is correctly defined for all $(x, y) \in X$, since the multifunction $G$ is
measurable. Moreover, $\varphi$ is nonnegative, and $\varphi(x, y) = 0$ iff $(x, y) \in M$. Denote by $\mathcal{F}$ the
feasible set of the problem $(\mathcal{P})$. Let us show that
\begin{equation} \label{eq:GlobalErrorBound}
  \varphi(x, y) \ge \dist\Big( (x, y), \mathcal{F} \Big) \quad 
  \forall x \in A, \: y \in \xL^p(\Omega, \mathfrak{A}, P; \mathbb{R}^m)
\end{equation}
Indeed, fix any $(x, y) \in X$ such that $x \in A$. If $\varphi(x, y) = + \infty$, then inequality
\eqref{eq:GlobalErrorBound} obviously holds true. Suppose now that $\varphi(x, y) < + \infty$. Then, in particular, one
has $G(x, \omega) \ne \emptyset$ for a.e. $\omega \in \Omega$.

By our assumptions the multifunction $G$ is measurable and has closed images. Therefore by 
\cite[Crlr.~8.2.13]{AubinFrankowska} there exists a measurable selection $z$ of the set-valued mapping $G(x, \cdot)$
such that 
\[
  | y(\omega) - z(\omega) | = \dist\big( y(\omega), G(x, \omega) \big) 
  \quad \text{ for a.e. } \omega \in \Omega.
\] 
Let us check that $z \in \xL^p(\Omega, \mathfrak{A}, P; \mathbb{R}^m)$. Then $(x, z) \in \mathcal{F}$ and
\[
  \varphi(x, y) = \| y - z \|_p = \big\| (x, y) - (x, z) \big\| \ge \dist\Big( (x, y), \mathcal{F} \Big),
\]
that is, inequality \eqref{eq:GlobalErrorBound} holds true.

To verify that $z$ belongs to the space $\xL^p$, observe that
\[
  |z(\omega)| \le |y(\omega)| + |z(\omega) - y(\omega)| = |y(\omega)| + \dist\big( y(\omega), G(x, \omega) \big)
\]
for a.e. $\omega \in \Omega$. The right-hand side of this inequality belongs to 
$\xL^p(\Omega, \mathfrak{A}, P; \mathbb{R}^m)$ due to the fact that $\varphi(x, y) < + \infty$. Therefore the function
$z$ belongs to this space as well. 

Thus, inequality \eqref{eq:GlobalErrorBound} holds true. Since the functional $\mathcal{I}$ is Lipschitz continuous on
bounded sets, by Lemma~\ref{lem:LipschitzFuncLowerEstimate} for any $R > 0$ there exists $L > 0$ such that
\[
  \mathcal{I}(x, y) \ge \mathcal{I}_* - L \dist\Big( (x, y), \mathcal{F} \Big) 
  \quad \forall (x, y) \in B(0, R).
\]
Hence by \cite[Prp.~3.16 and Remark~15, part (ii)]{Dolgopolik_ExPenFunc} the penalty function $\Phi_c$ is globally
exact.
\end{proof}

\begin{remark}
Note that by Corollary~\ref{crlr:LipschitzCont} the functional $\mathcal{I}$ is Lipschitz continuous on bounded sets in
the case $p > 1$, provided the integrand $f$ satisfies Assumption~1. In turn, as one can
readily verify, the set $\{ (x, y) \in X \mid x \in A, \: \Phi_c(x, y) < \mathcal{I}_* \}$ is bounded for some 
$c \ge 0$, if $1 \le p < + \infty$ and one of the following conditions is satisfied:
\begin{enumerate}
\item{the set $A$ is bounded, and the multifunction $G$ is bounded on $A \times \Omega$;
}

\item{the set $A$ is bounded, and there exist $C > 0$ and $\beta \in \xL^1(\Omega, \mathfrak{A}, P)$ such
that $f(x, y, \omega) \ge C |y|^p + \beta(\omega)$ for all $(x, y) \in A \times \mathbb{R}^{m}$ and a.e. 
$\omega \in \Omega$;
}

\item{the multifunction $G$ is bounded on $A \times \Omega$, and there exist $\beta \in \xL^1(\Omega, \mathfrak{A}, P)$
and a function $\rho \colon [0, + \infty) \to [0, +\infty)$ such that $\rho(t) \to + \infty$ as $t \to + \infty$, and
$f(x, y, \omega) \ge \rho(|x|) + \beta(\omega)$ for all $(x, y) \in \mathbb{R}^{d + m}$ and a.e. 
$\omega \in \Omega$;
}

\item{there exist $C > 0$, $\beta \in \xL^1(\Omega, \mathfrak{A}, P)$, and a function 
$\rho \colon [0, + \infty) \to [0, +\infty)$ such that $\rho(t) \to + \infty$ as $t \to + \infty$, and
$f(x, y, \omega) \ge \rho(|x|) + C |y|^p + \beta(\omega)$ for all $(x, y) \in \mathbb{R}^{d + m}$ and a.e. 
$\omega \in \Omega$;
}

\item{$(\Omega, \mathfrak{A}, P)$ is a finite probability space, and 
$\min_{\omega \in \Omega} f(x, y, \omega) \to + \infty$ as $|x| + |y| \to + \infty$.
}
\end{enumerate}
In the case $p = + \infty$ the set $\{ (x, y) \in X \mid x \in A, \: \Phi_c(x, y) < \mathcal{I}_* \}$ is bounded,
provided the first, the third or the last of the assumptions above is satisfied.
\end{remark}

In most particular cases the feasible set $G(x, \omega)$ of the second stage problem \eqref{prob:SecondStage} is not
defined explicitly, but rather via some constraints. As a result, one usually does not know an explicit expression for 
the penalty term $\varphi$ from Theorem~\ref{thrm:ExactPenaltyDistToSet}, which makes this theorem inapplicable to
real-world problems, at least in a direct way. In some cases Theorem~\ref{thrm:ExactPenaltyDistToSet} can still be
applied indirectly to reduce an analysis of the exactness of a penalty function for the problem $(\mathcal{P})$ to an
analysis of constraints of the second stage problem. Let us explain this statement with the use of a simple example.

\begin{example}
Suppose that the set-valued map $G$ is defined in the following way:
\[
  G(x, \omega) = \Big\{ y \in \mathbb{R}^m \Bigm| 0 \in Q(x, y, \omega) \Big\}
\]
where $Q \colon \mathbb{R}^d \times \mathbb{R}^m \times \Omega \to \mathbb{R}^s$ is a multifunction with closed images.
In other words, the second stage problem \eqref{prob:SecondStage} has the form:
\[
  \min_y \: f(x, y, \omega) \quad \text{subject to} \quad 0 \in Q(x, y, \omega).
\]
In this case it is natural to define
\[
  \varphi(x, y) = \Big( \mathbb{E}\big[ \dist(0, Q(x, y(\cdot), \cdot))^p \big] \Big)^{1/p},
  \quad 1 \le p < + \infty.
\]
Then $\varphi(x, y) = 0$ iff $(x, y) \in M$. Suppose that there exists $K > 0$ such that
\[
  K \dist(0, Q(x, y, \omega)) \ge \dist(y, G(x, \omega)) \quad 
  \forall x \in A, \: y \in \mathbb{R}^m, \: \omega \in \Omega,
\]
that is, the function $g(y) = \dist(0, Q(x, y, \omega))$ admits a global error bound uniform for all $x \in A$ and
$\omega \in \Omega$. Then
\[
  \Phi_{Kc}(x, y) = \mathcal{I}(x, y) + K c \varphi(x, y) \ge \mathcal{I}(x, y) + c \psi(x, y)
\]
for all $x \in A$ and $y \in \xL^p(\Omega, \mathfrak{A}, P; \mathbb{R}^m)$, where 
\[
  \psi(x, y) = \Big( \mathbb{E}[\dist(y(\cdot), G(x, \cdot))^p] \Big)^{1/p}.
\]
Therefore, as one can readily verify (cf. \cite[Prp.~2.2]{Dolgopolik_ExPenFunc}), under the assumptions of
Theorem~\ref{thrm:ExactPenaltyDistToSet} the penalty function $\Phi_c$ is globally exact and its least exact penalty
parameter is at most $K$ times greater than the least exact penalty parameter of the penalty function from
Theorem~\ref{thrm:ExactPenaltyDistToSet}.
\end{example}

Let us also point out two simple cases when Theorem~\ref{thrm:ExactPenaltyDistToSet} can be applied directly, that
is, the cases when one can write a simple explicit expression for the penalty term $\varphi$ from this theorem. Note
that Theorem~\ref{thrm:ExactPenaltyDistToSet} can be applied directly whenever the distance from a given point $y$ to
the set $G(x, \omega)$ is easy to compute, e.g. when the set $G(x, \omega)$ is defined by linear or, more generally,
convex quadratic constraints.

\begin{example}
Let $I := \{ 1, \ldots, m \}$. Suppose that the set $G(x, \omega)$ is defined by bound (box) constraints, that is,
\[
  G(x, \omega) = \Big\{ y = (y_1, \ldots, y_m)^T \in \mathbb{R}^m \Bigm| a_i(x, \omega) \le y_i \le b_i(x, \omega), 
  \enspace i \in I \Big\}
\]
for some given functions $a_i$ and $b_i$. Let the space $\mathbb{R}^m$ be equipped with the $\ell^{\infty}$ norm. Then
the penalty term $\varphi$ from Theorem~\ref{thrm:ExactPenaltyDistToSet} has the form
\[
  \varphi(x, y) = \Big( \int_{\Omega} 
  \max_{i \in I} \big\{ 0, y_i(\omega) - b_i(x, \omega), a_i(x, \omega) - y_i(\omega) \big\}^p  \, d P(\omega) 
  \Big)^{1/p}
\]
in the case $1 \le p < + \infty$.
\end{example}

\begin{example}
Let $G(x, \omega) = B(z(x, \omega), R(x, \omega))$ be the closed ball with centre $z(x, \omega)$ and radius 
$R(x, \omega)$. Then the penalty term $\varphi$ from Theorem~\ref{thrm:ExactPenaltyDistToSet} has the form
\[
  \varphi(x, y) = \Big( \int_{\Omega} 
  \max \big\{ 0, |y(\omega) - z(x, \omega)| - R(x, \omega) \big\}^p  \, d P(\omega) \Big)^{1/p}.
\]
in the case $1 \le p < + \infty$.
\end{example}

Observe that the penalty terms from Theorem~\ref{thrm:ExactPenaltyDistToSet} and the examples above depend on 
the parameter $p$ that defines the space in which one solves the problem $(\mathcal{P})$. This parameter must be chosen
to satisfy the assumption of Theorem~\ref{thrm:ExactPenaltyDistToSet}. 

Under some additional assumptions on constraints of the second stage problem one can prove the global exactness of 
the penalty function $\Phi_c$ with a penalty term $\varphi$ that does not depend on $p$. For the sake of simplicity, we
will prove this result only in the case when the feasible set $G(x, \omega)$ of the second stage problem is defined by
inequality constraints, i.e. it has the form
\[
  G(x, \omega) = \Big\{ y \in \mathbb{R}^m \Bigm| g_i(x, y, \omega) \le 0, \: i \in I = \{ 1, \ldots, \ell \} \Big\}
\]
for some functions $g_i \colon \mathbb{R}^d \times \mathbb{R}^m \times \Omega \to \mathbb{R}$. Below we suppose that
for each $x \in \mathbb{R}^d$ the map $(y, \omega) \mapsto g_i(x, y, \omega)$, $i \in I$, is a Carath\'{e}odory
function, so that the penalty term
\begin{equation} \label{eq:L1PenTerm}
  \varphi(x, y) = \int_{\Omega} \max_{i \in I}\big\{ 0, g_i(x, y, \omega) \big\} \, d P(\omega)
\end{equation}
is correctly defined. Note that $\varphi(x, y) = 0$ iff $(x, y) \in M$. We will assume that for any 
$x \in \mathbb{R}^d$ and a.e. $\omega \in \Omega$ the function $y \mapsto g_i(x, y, \omega)$, $i \in I$, is
quasidifferentiable and denote by 
$\mathscr{D}_y g_i(x, y, \omega) = [\underline{\partial}_y g_i(x, y, \omega), \overline{\partial}_y g_i(x, y, \omega)]$
its quasidifferential. Denote also
$I(x, y, \omega) = \{ i \in I \mid g_i(x, y, \omega) = \max_{k \in I} g_k(x, y, \omega) \}$.

Let $(Y, d)$ be a metric space, $K \subset Y$ be a given set, and $g \colon Y \to \mathbb{R} \cup \{ + \infty \}$
be a given function. Recall that for any $y \in K \cap \domain g$ the quantity
\[
  g^{\downarrow}_K(y) = \liminf_{z \to y, z \in K} \frac{g(z) - g(y)}{d(z, y)}
\]
is called \textit{the rate of steepest descent} of $g$ at $y$. If $y$ is not a limit point of the set $K$, then by
definition $g^{\downarrow}_K(y) = + \infty$. Recall also that a point  $y \in K \cap \domain g$ is called 
\textit{an inf-stationary} point of $g$ on the set $K$, if $g^{\downarrow}_K(y) \ge 0$. It should be noted
that in various particular cases this inequality is reduced to standard stationarity conditions. For example, if $Y$ is
normed space, $g$ is Fr\'{e}chet differentiable at a point $y \in K$, and the set $K$ is convex, then 
$g^{\downarrow}_K(y) \ge 0$ iff $g'(y)[z - y] \ge 0$ for all $z \in K$, where $g'(y)$ is the Fr\'{e}chet derivative of
$g$ at $y$. See \cite{Demyanov2000,Demyanov2010,Uderzo,Uderzo2} for more details on the rate of steepest descent and
the definition of inf-stationarity.

\begin{theorem} \label{thrm:ExactPenaltyL1}
Let $1 \le p < + \infty$ and the following assumptions be valid:
\begin{enumerate}
\item{there exist a globally optimal solution of the problem $(\mathcal{P})$;
\label{assumpt:GlobalMinExistence}}

\item{the functional $\mathcal{I}$ is Lipschitz continuous on bounded sets;}

\item{the set $S_c(\gamma) = \{ (x, y) \in X \mid x \in A, \: \Phi_c(x, y) < \gamma \}$ is bounded for some $c \ge 0$
and $\gamma > \mathcal{I}_*$, where $\Phi_c$ is the penalty functions with the penalty term \eqref{eq:L1PenTerm};
\label{assumpt:SublevelBoundedness}}

\item{for any $x \in A$ there exists an a.e. nonnegative function $L(\cdot) \in \xL^1(\Omega, \mathfrak{A}, P)$ such
that
$|g_i(x, y_1, \omega) - g_i(x, y_2, \omega)| \le L(\omega) \| y_1 - y_2 \|$ for all $y_1, y_2 \in \mathbb{R}^d$, all
$i \in I$ and a.e. $\omega \in \Omega$;
\label{assumpt:GlobalLipschitz}}

\item{for all $i \in I$, $x \in A$, and $y \in \xL^p(\Omega, \mathfrak{A}, P; \mathbb{R}^m)$ the set-valued mappings
$\underline{\partial}_y g_i(x, y(\cdot), \cdot)$ and $\overline{\partial}_y g_i(x, y(\cdot), \cdot)$ are measurable;
\label{assumpt:MeasQuasidiff}}

\item{there exists $a > 0$ such that for any $(x, y) \in A \times \mathbb{R}^{m}$ and a.e. $\omega \in \Omega$ such
that $y \notin G(x, \omega)$, and for all $i \in I(x, y, \omega)$ one can find 
$w_i(x, y, \omega) \in \overline{\partial}_y g_i(x, y, \omega)$ satisfying the following condition:
\begin{equation} \label{eq:ConstraintsNondeg}
  \dist\Big( 0, \co\Big\{ \underline{\partial}_y g_i(x, y, \omega) + w_i(x, y, \omega) \Bigm| 
  i \in I(x, y, \omega) \Big\} \Big) \ge a.
\end{equation}
\label{assumpt:GlobalCQ}}
\end{enumerate}
\vspace{-5mm}
Then the penalty function $\Phi_c$ with the penalty term \eqref{eq:L1PenTerm} is globally exact and
there exists $c_* \ge 0$ such that for any $c \ge c_*$ the following statements hold true:
\begin{enumerate}
\item{$(x_*, y_*) \in S_c(\gamma)$ is a locally optimal solution of the penalized problem \eqref{prob:PenalizedProblem}
iff $(x_*, y_*)$ is a locally optimal solution of the problem $(\mathcal{P})$;
}

\item{$(x_*, y_*) \in S_c(\gamma)$ is an inf-stationary point of the penalty function $\Phi_c$ on the set
$A \times \xL^p(\Omega, \mathfrak{A}, P; \mathbb{R}^m)$ iff $(x_*, y_*)$ is an inf-stationary point of the functional
$\mathcal{I}$ on the feasible set $\mathcal{F}$ of the problem $(\mathcal{P})$.
}
\end{enumerate}
\end{theorem}

\begin{proof}
Let us show that under the assumptions of the theorem $\varphi^{\downarrow}(x, \cdot)(y) \le - a$ for any 
$(x, y) \in X \setminus \mathcal{F}$ such that $x \in A$ and $\varphi(x, y) < + \infty$ 
(here $\varphi^{\downarrow}(x, \cdot)(y)$ is the rate of steepest descent of the function $y \mapsto \varphi(x, y)$ at
the point $y$). Then applying \cite[Thrm.~2]{DolgopolikFominyh} one obtains the required result.

To prove the required estimate for $\varphi^{\downarrow}(x, \cdot)(y)$, we first construct a descent direction for the
function $\varphi$ using condition \eqref{eq:ConstraintsNondeg}, and then obtain an upper estimate for the rate of
steepest descent via the directional derivative of $\varphi$ along the constructed descent direction.

Fix any $(x, y) \in X \setminus \mathcal{F}$ such that $x \in A$ and $\varphi(x, y) < + \infty$. Recall that by 
the definition of quasidifferential one has
\begin{equation} \label{eq:ConstrQuasidiff}
  Q_i(h, \omega) = (g_i(x, \cdot, \omega))'(y(\omega), h) = 
  \max_{v \in \underline{\partial}_y g_i(x, y(\omega), \omega)} \langle v, h \rangle
  + \min_{w \in \overline{\partial}_y g_i(x, y(\omega), \omega)} \langle w, h \rangle
\end{equation}
(see Def.~\ref{def:Quasidifferential}). Applying Assumption~\ref{assumpt:MeasQuasidiff} and
\cite[Thrm.~8.2.11]{AubinFrankowska} one obtains that the function $Q_i$ is measurable in $\omega$ for any 
$h \in \mathbb{R}^m$. Moreover, since in the finite dimensional case the quasidifferential is a pair of compact convex
sets, the function $Q_i$ is continuous in $h$ for a.e. $\omega \in \Omega$, i.e. $Q_i$ is a Carath\'{e}odory function.

Let us now prove that the multifunction $I(\cdot) := I(x, y(\cdot), \cdot)$, $I \colon \Omega \to \{ 1, \ldots, \ell \}$
is measurable. Indeed, by definitions for any nonempty subset $K \subseteq \{ 1, \ldots, \ell \}$ one has
\begin{align*}
  I^{-1}(K) &= \Big\{ \omega \in \Omega \Bigm| I(x, y(\omega), \omega) \cap K \ne \emptyset \Big\}
  \\
  &= \Big\{ \omega \in \Omega \Bigm| 
  \max_{k \in K} g_k(x, y(\omega), \omega) \ge \max_{i \in I} g_i(x, y(\omega), \omega) \Big\}.
\end{align*}
This set is measurable, since the functions $g_i(x, y(\cdot), \cdot))$ are measurable due to the fact that the maps
$(y, \omega) \mapsto g_i(x, y, \omega)$ are Carath\'{e}odory functions by our assumption. Thus, for any subset 
$K \subseteq \{ 1, \ldots, s \}$ the set $I^{-1}(K)$ is measurable, that is, the set-valued map
$I(\cdot)$ is measurable by definition (see, e.g. \cite[Def.~8.1.1]{AubinFrankowska}).

Introduce the set
\[
  E = \Big\{ \omega \in \Omega \Bigm| \max_{i \in I} g_i(x, y(\omega), \omega) > 0 \Big\}.
\]
Note that the set $E$ is measurable, thanks to our assumption that the mappings $(y, \omega) \mapsto g_i(x, y, \omega)$
are Carath\'{e}odory functions. Moreover, $P(E) > 0$ due to the fact that $(x, y)$ is not a feasible point of 
the problem $(\mathcal{P})$.

Since the multifunction $I(\cdot)$ is measurable and $Q_i$ are Carath\'{e}odory functions, the set-valued
mapping
\[
  H(\omega) := \Big\{ h \in \mathbb{R}^m \Bigm| |h| = 1, \:
  \max_{i \in I(\omega)} Q_i(h, \omega) 
  = \min_{|z| = 1} \max_{i \in I(\omega)} Q_i(z, \omega) \Big\}, \quad \omega \in E
\]
is measurable by \cite[Thrm.~8.2.11]{AubinFrankowska}. Furthermore, this multifunction obviously has closed images.
Therefore by \cite[Thrm.~8.1.3]{AubinFrankowska} there exists a measurable function
$h_* \colon E \to \mathbb{R}^m$ such that $h_*(\omega) \in H(\omega)$ for all $\omega \in E$. For any
$\omega \in \Omega \setminus E$ define $h_*(\omega) = 0$. Then $h_* \colon \Omega \to \mathbb{R}^m$ is a measurable
function and, moreover, $\| h_* \|_p = P(E) > 0$. 

From condition \eqref{eq:ConstraintsNondeg} and the separation theorem it follows that for any $\omega \in E$ there
exists $\widehat{h}(\omega) \in \mathbb{R}^m$ with $|\widehat{h}(\omega)| = 1$ such that
\[
  \langle v, \widehat{h}(\omega) \rangle \le - a \quad 
  \forall v \in \co\Big\{ \underline{\partial}_y g_i(x, y(\omega), \omega) + w_i(x, y(\omega), \omega) \Bigm| 
  i \in I(\omega) \Big\}.
\]
Hence with the use of \eqref{eq:ConstrQuasidiff} one obtains that $Q_i(\widehat{h}(\omega), \omega) \le - a$ for all
$\omega \in E$ and $i \in I(\omega)$, which by the definition of $h_*$ implies that
\begin{equation} \label{eq:DecayDirection}
  \max_{i \in I(\omega)} Q_i(h_*(\omega), \omega) 
  \begin{cases}
    \le - a, & \text{ if } \omega \in E, 
    \\
    = 0, & \text{ if } \omega \notin E.
  \end{cases}
\end{equation}
Thus, the function $h_*$ is the desired descent direction, along which we will evaluate the directional derivative of 
the penalty term $\varphi$.

Indeed, denote $\psi(\omega, \alpha) = \max_{i \in I} \{ 0, g_i(x, y(\omega) + \alpha h_*(\omega), \omega) \}$ for all
$\alpha \ge 0$ and $\omega \in \Omega$. Applying relations \eqref{eq:DecayDirection} and standard calculus rules for
directional derivatives (see, e.g. \cite{DemyanovRubinov}) one gets that
\[
  \lim_{\alpha \to + 0} \frac{\psi(\omega, \alpha) - \psi(\omega, 0)}{\alpha} =
  \begin{cases}
    \max_{i \in I(\omega)} Q_i(h_*(\omega), \omega) \le - a, & \text{ if } \omega \in E,
    \\
    0, & \text{ if } \omega \notin E.
  \end{cases}
\]
Applying Assumption~\ref{assumpt:GlobalLipschitz} and the well-known fact that the maximum of a finite family of
Lipschitz continuous is Lipschitz continuous (see, e.g. \cite[Appendix~III]{DemyanovMalozemov}) one obtains that 
there exists an a.e. nonnegative function $L(\cdot) \in \xL^1(\Omega, \mathfrak{A}, P)$ such that
\[
  \left| \frac{\psi(\omega, \alpha) - \psi(\omega, 0)}{\alpha} \right|
  \le L(\omega) |h_*(\omega)| \le L(\omega) \quad \forall \alpha > 0, \quad \text{a.e. } \omega \in \Omega.
\]	
Note also that $\psi(\cdot, 0) \in \xL^1(\Omega, \mathfrak{A}, P)$, since $\varphi(x, y) < + \infty$. Hence by the
inequality above $\psi(\cdot, \alpha) \in \xL^1(\Omega, \mathfrak{A}, P)$ for all $\alpha > 0$. Consequently, applying
Lebesgue's dominated convergence theorem and the fact that 
$\varphi(x, y + \alpha h_*) = \mathbb{E}[\psi(\cdot, \alpha)]$ one obtains that
\begin{align*}
  \big[ \varphi(x, \cdot) \big]'(y; h_*) 
  &= \lim_{\alpha \to + 0} \frac{\varphi(x, y + \alpha h_*) - \varphi(x, y)}{\alpha}
  \\
  &= \int_E \max_{i \in I(\omega)} Q_i(h_*(\omega), \omega) \, d P(\omega) \le - a P(E).
\end{align*}
Therefore
\begin{multline*}
  \varphi^{\downarrow}(x, \cdot)(y) = \liminf_{z \to y} \frac{\varphi(x, z) - \varphi(x, y)}{\| z - y \|_p}
  \\
  \le \liminf_{\alpha \to +0} \frac{\varphi(x, y + \alpha h_*) - \varphi(x, y)}{\alpha \| h_* \|_p}
  = \frac{\big[ \varphi(x, \cdot) \big]'(y; h_*)}{\| h_* \|_p}
  \le - \frac{a P(E)}{P(E)} = -a,
\end{multline*}
and the proof is complete.
\end{proof}

\begin{remark}
{(i)~Note that by \cite[Crlr.~14.14]{RockafellarWets} the multifunctions 
$\underline{\partial}_y g_i(x, y(\cdot), \cdot)$ and $\overline{\partial}_y g_i(x, y(\cdot), \cdot)$ are measurable for
any measurable function $y(\cdot)$, provided for any $\omega \in \Omega$ the mapping 
$\underline{\partial}_y g_i(x, \cdot, \omega)$ is outer semicontinuous and the graphical mapping
$\omega \mapsto \graph \underline{\partial}_y g_i(x, \cdot, \omega)$ is measurable.
}

\noindent{(ii)~In the case when the functions $g_i$ are continuously differentiable in $y$, assumption
\eqref{eq:ConstraintsNondeg} is satisfied iff there exists $a > 0$ such that for any $(x, y) \in \mathbb{R}^{d + m}$ and
a.e. $\omega \in \Omega$ such that $y \notin G(x, \omega)$ one has
\[
  \dist\Big( 0, \co\Big\{ \nabla_y g_i(x, y, \omega) \Bigm| i \in I(x, y, \omega) \Big\} \Big) \ge a.
\]
This condition can be viewed as a uniform Mangasarian-Fromovitz constraint qualification. In turn, in the case when the
functions $g_i$ are convex in $y$, assumption \eqref{eq:ConstraintsNondeg} is satisfied iff there exists $a > 0$ such
that for any $(x, y) \in \mathbb{R}^{d + m}$ and a.e. $\omega \in \Omega$ such that $y \notin G(x, \omega)$ one has
\[
  \dist\Big( 0, \co\Big\{ \partial_y g_i(x, y, \omega) \Bigm| i \in I(x, y, \omega) \Big\} \Big) \ge a.
\]
where $\partial_y g_i(x, y, \omega)$ is the subdifferential of the function $g_i(x, \cdot, \omega)$ in the sense of
convex analysis.
}
\end{remark}

\begin{remark} \label{rmrk:ConnectionToDCProblems}
Let for a.e. $\omega \in \Omega$ the functions $(x, y) \mapsto f(x, y, \omega)$ and 
$(x, y) \mapsto g_i(x, y, \omega)$, $i \in I$, be DC (Difference-of-Convex), that is, there exist convex in $(x, y)$
functions $f_1(x, y, \omega), f_2(x, y, \omega), g_{i1}(x, y, \omega)$, and $g_{i2}(x, y, \omega)$ such that
\[
  f(x, y, \omega) = f_1(x, y, \omega) - f_2(x, y, \omega), \quad
  g_i(x, y, \omega) = g_{i1}(x, y, \omega) - g_{i2}(x, y, \omega)
\]
for all $(x, y) \in \mathbb{R}^{d + m}$, $i \in I$, and a.e. $\omega \in \Omega$. Then the penalty function from
Theorem~\ref{thrm:ExactPenaltyL1} is DC as well. Namely, one has $\Phi_c(x, y) = \Phi_c^1(x, y) - \Phi_c^2(x, y)$, where
\begin{align*}
  \Phi_c^1(x, y) = \int_{\Omega} \Big( &f_1(x, y(\omega), \omega) 
  \\
  &+ c \max_{i \in I} \Big\{ 0, g_{i1}(x, y(\omega), \omega) + \sum_{k \ne i} g_{k2}(x, y(\omega), \omega) \Big\}
  \Big) \, d P(\omega),
\end{align*}
and
\[
  \Phi_c^2(x, y) = \int_{\Omega} \Big( f_2(x, y(\omega), \omega) 
  + c \sum_{i \in I} g_{i2}(x, y(\omega), \omega) \Big) \, d P(\omega)
\]
are convex functionals. Therefore with the use of Theorem~\ref{thrm:ExactPenaltyL1} and well-known global optimality
conditions for DC optimization problems one can easily obtain global optimality conditions for the problem
$(\mathcal{P})$ and the original two-stage stochastic programming problem (cf.~\cite{Strekalovsky}). Moreover, under
the assumptions of Theorem~\ref{thrm:ExactPenaltyL1} one can apply well-developed methods of DC optimization to find
local or global minima of the DC penalty function $\Phi_c(x, y)$, which coincide with local/global minima of the problem
$(\mathcal{P})$. Thus, Theorem~\ref{thrm:ExactPenaltyL1} opens a way for applications of DC programming algorithms to
two-stage stochastic programming problems (cf.~\cite{StrekalovskyOrlov,Orlov}).
\end{remark}

\subsection{Optimality Conditions}
\label{subsect:OptimalityConditions}

Let us finally derive optimality conditions for the problem $(\mathcal{P})$ in terms of codifferentials. We will derive
these conditions by applying standard optimality conditions for quasidifferentiable functions to an exact penalty
function for the problem $(\mathcal{P})$. 

For the sake of shortness, we will consider only the case when the set $A$ is convex and obtain optimality conditions
under the assumptions of Theorem~\ref{thrm:ExactPenaltyL1}. It should be noted that one can obtain such conditions under
less restrictive assumptions on the functional $\mathcal{I}$ and the penalty function $\Phi_c$, if one considers the
so-called \textit{local} exactness of the penalty function instead of the global one
(see~\cite{Demyanov2010,Dolgopolik_ExPenFunc}). Moreover, one can significantly relax the assumptions on the constraints
of the second-stage problem by considering the case $p = + \infty$ and utilising the highly nonsmooth penalty term
\[
  \varphi(x, y) = \esssup_{\omega \in \Omega} \Big\{ \max_{i \in I}\{ 0, g_i(x, y(\omega), \omega) \} \Big\}.
\]
However, the price one has to pay for less restrictive assumptions on constraints is the reduced regularity of Lagrange
multipliers (see the theorem below). Namely, in this case one must assume that the Lagrange multipliers are just
finitely additive measures. 

For any convex subset $K$ of a Banach space $Y$ and any $y \in K$ denote the normal cone to the set $K$ at
the point $y$ by $N_K(y) = \{ y^* \in Y^* \mid \langle y^*, z - y \rangle \le 0 \: \forall z \in K \}$.

\begin{theorem} \label{thrm:OptCond}
Let $1 < p < + \infty$, the set $A$ be convex, the feasible set of the second-stage problem \eqref{prob:SecondStage}
have the form
\[
  G(x, \omega) = \Big\{ y \in \mathbb{R}^m \Bigm| g_i(x, y, \omega) \le 0, \: i \in I = \{ 1, \ldots, \ell \} \Big\}
\]
for some functions $g_i \colon \mathbb{R}^d \times \mathbb{R}^m \times \Omega \to \mathbb{R}$, the function $f$ satisfy
Assumption~1, and the functions $g_i$, $i \in I$, satisfy the same assumption. Suppose
also that assumptions \ref{assumpt:GlobalMinExistence}, \ref{assumpt:SublevelBoundedness}--\ref{assumpt:GlobalCQ} of
Theorem~\ref{thrm:ExactPenaltyL1} are valid, and $(x_*, y_*)$ is a locally optimal solution of the problem 
$(\mathcal{P})$ such that $(x_*, y_*) \in S_c(\gamma)$ for some $c \ge c_*$, where $c_*$ is from 
Theorem~\ref{thrm:ExactPenaltyL1}. 

Then for any measurable selection $(0, w_x(\cdot), w_y(\cdot))$ of 
the set-valued mapping $\overline{d}_{x, y} f(x_*, y_*(\cdot), \cdot)$ and any measurable selections 
$(0, w_{xi}(\cdot), w_{yi}(\cdot))$ of the multifunctions $\overline{d}_{x, y} g_i(x_*, y_*(\cdot), \cdot)$, $i \in I$,
there exist $\zeta \in \xL^1(\Omega, \mathfrak{A}, P; \mathbb{R}^d)$ and nonnegative multipliers 
$\lambda_i \in \xL^{\infty}(\Omega, \mathfrak{A}, P)$, $i \in I$, such that $\mathbb{E}[\zeta] \in -N_A(x_*)$,
$\sum_{i \in I} \| \lambda_i \|_{\infty} \le c_*$, $\lambda_i(\omega) g_i(x_*, y_*(\omega), \omega) = 0$ for a.e.
$\omega \in \Omega$ and all $i \in I$, and
\begin{align*}
  (0, \zeta(\omega), 0) &\in \underline{d}_{x, y} f(x_*, y_*(\cdot), \cdot) + (0, w_x(\cdot), w_y(\cdot))
  \\
  &+ \sum_{i = 1}^{\ell} 
  \lambda_i(\omega) \Big( \underline{d}_{x, y} g_i(x_*, y_*(\cdot), \cdot) + (0, w_{xi}(\cdot), w_{yi}(\cdot)) \Big)
\end{align*}
for a.e. $\omega \in \Omega$.
\end{theorem}

\begin{proof}
Under the assumptions of the theorem the functional $\mathcal{I}$ is Lipschitz continuous on bounded sets by
Corollary~\ref{crlr:LipschitzCont}. Let 
\[
  \varphi(x, y) = \int_{\Omega} \max_{i \in I}\big\{ 0, g_i(x, y, \omega) \big\} \, d P(\omega)
  \quad \forall (x, y) \in X.
\]
Then by Theorem~\ref{thrm:ExactPenaltyL1} the pair $(x_*, y_*)$ is a point of local minimum of the penalty function
$\Phi_c$ on the set $A \times \xL^p(\Omega, \mathfrak{A}, P)$ for any $c \ge c_*$, where $c_*$ is from 
Theorem~\ref{thrm:ExactPenaltyL1}. Thus, in particular, $(x_*, y_*)$ is a point of local minimum of the problem
\[
  \min_{(x, y)} \: \mathcal{J}(x, y) = \int_{\Omega} f_0(x, y(\omega), \omega) \, d P(\omega)
  \quad \text{s.t.} \quad (x, y) \in A \times \xL^p(\Omega, \mathfrak{A}, P; \mathbb{R}^m),
\]
where $f_0 = f + c_* \max_{i \in I} \{ 0, g_i \}$. The function $f_0$ is codifferentiable in $(x, y)$, and applying
codifferential calculus (see, e.g. \cite{DemyanovRubinov}) one can compute its codifferential and verify that $f_0$
satisfies Assumption~1. Therefore by Corollary~\ref{crlr:Quasidifferentiability} the
functional $\mathcal{J}$ is directionally differentiable. Applying well-known necessary conditions for a minimum of a
directionally differentiable function on a convex set (see, e.g. \cite[Lemma~V.1.2]{DemyanovRubinov}) and
Corollary~\ref{crlr:Quasidifferentiability} one obtains that
\[
  \mathcal{J}'(x_*, y_*; h_x, h_y) 
  = \int_{\Omega} \big[ f_0(\cdot, \cdot, \omega) \big]'(x_*, y_*(\omega); h_x, h_y(\omega)) \, d P(\omega) \ge 0
\]
for all $(h_x, h_y) \in (A - x_*) \times \xL^p(\Omega, \mathfrak{A}, P; \mathbb{R}^m)$. Hence with the use of the
standard calculus rules for directional derivatives (see \cite[Sect.~I.3]{DemyanovRubinov}) one gets that for all
such $(h_x, h_y)$ the following inequality holds true:
\begin{multline*}
  \mathcal{J}'(x_*, y_*; h_x, h_y) 
  = \int_{\Omega} \Big( \big[ f(\cdot, \cdot, \omega) ]'(x_*, y_*(\omega); h_x, h_y(\omega)) 
  \\ 
  + c_* \max_{i \in \widehat{I}(\omega)} \big[ g_i(\cdot, \cdot, \omega) \big]'(x_*, y_*(\omega); h_x, h_y(\omega)) 
  \Big) \, d P(\omega) \ge 0,
\end{multline*}
where $g_0(x, y, \omega) \equiv 0$ and
\[
  \widehat{I}(\omega) = \Big\{ i \in I \cup \{ 0 \} \Bigm| 
  g_i(x_*, y_*(\omega), \omega) = \max_{i \in I} \big\{ 0, g_i(x_*, y_*(\omega), \omega) \big\} \Big\}.
\]
Fix a measurable selection $(0, w_x(\cdot), w_y(\cdot))$ of the set-valued map 
$\overline{d}_{x, y} f(x_*, y_*(\cdot), \cdot)$, for all $i \in I$ fix any measurable selections 
$(0, w_{xi}(\cdot), w_{yi}(\cdot))$ of the set-valued maps $\overline{d}_{x, y} g_i(x_*, y_*(\cdot), \cdot)$,
and denote $(w_{x0}(\cdot), w_{y0}(\cdot)) \equiv 0$. Then by the definition of quasidifferential
(Def.~\ref{def:Quasidifferential}) and equality \eqref{eq:QuasidiffViaCodiff} one has
\begin{multline*}
  \int_{\Omega} \Big( \max_{(v_x, v_y) \in \underline{\partial}_{x, y} f(x_*, y_*(\omega), \omega)}
  \big( \langle v_x + w_x(\omega), h_x \rangle + \langle v_y + w_y(\omega), h_y(\omega) \rangle \big)
  \\
  + c_* \max_{i \in \widehat{I}(\omega)} \max
  \big( \langle v_{xi} + w_{xi}(\omega), h_x \rangle + \langle v_{yi} + w_{yi}(\omega), h_y(\omega) \rangle \big)
  \Big) \, d P(\omega) \ge 0
\end{multline*}
for all $(h_x, h_y) \in (A - x_*) \times \xL^p(\Omega, \mathfrak{A}, P; \mathbb{R}^m)$, where the last maximum is
taken over all $(v_{xi}, v_{yi}) \in \underline{\partial}_{x, y} g_i(x_*, y_*(\omega), \omega)$. Consequently, one has
\begin{equation} \label{eq:OptCondViaUCA}
  \int_{\Omega} \max_{(v_x, v_y) \in Q(\omega)} 
  \big( \langle v_x, h_x \rangle + \langle v_y, h_y(\omega) \rangle \big) \, d P(\omega) \ge 0
\end{equation}
for all $(h_x, h_y) \in (A - x_*) \times \xL^p(\Omega, \mathfrak{A}, P; \mathbb{R}^m)$, where
\begin{multline*}
  Q(\omega) = \underline{\partial}_{x, y} f(x_*, y_*(\omega), \omega) + (w_x(\omega), w_y(\omega))
  \\
  + c_* \co\Big\{ \underline{\partial}_{x, y} g_i(x_*, y_*(\omega), \omega) + (w_{xi}(\omega), w_{yi}(\omega))
  \Bigm| i \in \widehat{I}(\omega) \Big\}
\end{multline*}
for any $\omega \in \Omega$. 

Let us show that the multifunction $Q(\cdot)$ is measurable. Indeed, as was pointed out in the proof of
Corollary~\ref{crlr:Quasidifferentiability}, Assumption~1 guarantees that 
the set-valued mappings $\underline{\partial}_{x, y} f(x_*, y_*(\cdot), \cdot)$ and 
$\underline{\partial}_{x, y} g_i(x_*, y_*(\cdot), \cdot)$, $i \in I \cup \{ 0 \}$, are measurable. Hence with the use of
\cite[Proposition~14.11, part (c)]{RockafellarWets} one gets that the set-valued mappings 
$\underline{\partial}_{x, y} f(x_*, y_*(\cdot), \cdot) + (w_x(\cdot), w_y(\cdot))$
and $\underline{\partial}_{x, y} g_i(x_*, y_*(\cdot), \cdot) + (w_{xi}(\cdot), w_{yi}(\cdot))$, 
$i \in I \cup \{ 0 \}$, are measurable as well.

Arguing in the same way as in the proof of Theorem~\ref{thrm:ExactPenaltyL1} one can easily check that the multifunction
$\widehat{I}(\cdot)$ is measurable, which implies that the set-valued maps
\[
  Q_i(\omega) := \begin{cases}
    \underline{\partial}_{x, y} g_i(x_*, y_*(\omega), \omega) + (w_{xi}(\omega), w_{yi}(\omega)), & \text{if } 
    i \in \widehat{I}(\omega),
    \\
    \emptyset, & \text{if } i \notin \widehat{I}(\omega)
  \end{cases}
\]
are measurable for all $i \in I \cup \{ 0 \}$. Therefore by \cite[Prp.~14.11, part (b)]{RockafellarWets} and 
\cite[Thrm.~8.2.2]{AubinFrankowska} the set-valued map 
\[
  \co \Big( \bigcup_{i \in I \cup \{ 0 \}} Q_i(\cdot) \Big) 
  = \co\Big\{ \underline{\partial}_{x, y} g_i(x_*, y_*(\cdot), \cdot) + \{ (w_{xi}(\cdot), w_{yi}(\cdot) \}
  \Bigm| i \in \widehat{I}(\cdot) \Big\}.
\]
is measurable. Hence applying \cite[Prp.~14.11, part (c)]{RockafellarWets} one finally gets that the multifunction
$Q(\cdot)$ is measurable.

Now, arguing in the same way as in the proof of Lemma~\ref{lem:AlmostCodifferentiability} (or utilising the
interchangeability principle; see, e.g. \cite[Thrm.~14.60]{RockafellarWets}) one gets that inequality
\eqref{eq:OptCondViaUCA} is satisfied iff
\begin{equation} \label{eq:OptCondViaSelections}
  \max_{(v_x(\omega), v_y(\cdot))}
  \int_{\Omega} \big( \langle v_x(\omega), h_x \rangle + \langle v_y(\omega), h_y(\omega) \rangle \big) 
  \, d P (\omega) \ge 0
\end{equation}
forall $(h_x, h_y) \in (A - x_*) \times \xL^p(\Omega, \mathfrak{A}, P; \mathbb{R}^m)$, where the maximum is taken over
all measurable selections of the multifunction $Q(\cdot)$ (note that at least one such selection exists by
\cite[Thrm.~8.1.3]{AubinFrankowska}). From the definition of $Q(\cdot)$ and the growth condition on the codifferentials
of the functions $f$ and $g_i$ from Assumption~1 it follows that the set of all
measurable selection of $Q(\cdot)$ is a bounded subspace of the space
$\xL^1(\Omega, \mathfrak{A}, P; \mathbb{R}^d) \times \xL^{p'}(\Omega, \mathfrak{A}, P; \mathbb{R}^m)$. Therefore
inequality \eqref{eq:OptCondViaSelections} can be rewritten as follows:
\begin{equation} \label{eq:OptCondInequal}
  \max_{(v_1, v_2) \in \mathcal{Q}(x_*, y_*)} \Big( \langle v_1, h_x \rangle +
  \int_{\Omega} \langle v_y(\omega), h_y(\omega) \rangle \, d P(\omega) \Big) \ge 0
\end{equation}
for all $(h_x, h_y) \in (A - x_*) \times \xL^p(\Omega, \mathfrak{A}, P; \mathbb{R}^m)$, where
\begin{multline*}
  \mathcal{Q}(x_*, y_*) := \Big\{ (v_1, v_2) \in \mathbb{R}^d \times \xL^{p'}(\Omega, \mathfrak{A}, P; \mathbb{R}^m)
  \Bigm| v_1 = \mathbb{E}[v_x], \quad v_2 = v_y, 
  \\ 
  (v_x(\cdot), v_y(\cdot))
  \text{ is a measurable selection of the map } Q(\cdot) \Big\}.
\end{multline*}
The set $\mathcal{Q}(x_*, y_*)$ is bounded due to the boundedness of the set of all measurable selections of
$Q(\cdot)$. Furthermore, the set $\mathcal{Q}(x_*, y_*)$ is convex and closed, since by definition $Q(\cdot)$ has closed
and convex images. Therefore, $\mathcal{Q}(x_*, y_*)$ is a weakly compact convex subset of
$\mathbb{R}^d \times \xL^{p'}(\Omega, \mathfrak{A}, P; \mathbb{R}^m)$. Hence taking into account inequality
\eqref{eq:OptCondInequal} and applying the separation theorem one can easily check that
\[
  \mathcal{Q}(x_*, y_*) \cap \Big( \big\{ - N_A(x_*) \} \times \{ 0 \} \Big) \ne \emptyset.
\]
Consequently, by the definitions of $\mathcal{Q}(x_*, y_*)$ and $Q(\cdot)$ there exists a function 
$\zeta \in \xL^1(\Omega, \mathfrak{A}, P; \mathbb{R}^d)$ such that $\mathbb{E}[\zeta] \in -N_A(x_*)$ and
\begin{equation} \label{eq:OptCondDualForm}
\begin{split}
  (\zeta(\omega), 0) &\in \underline{\partial}_{x, y} f(x_*, y_*(\omega), \omega) + (w_x(\omega), w_y(\omega))
  \\
  &+ c_* \co\Big\{ \underline{\partial}_{x, y} g_i(x_*, y_*(\omega), \omega) + (w_{xi}(\omega), w_{yi}(\omega))
  \Bigm| i \in \widehat{I}(\omega) \Big\}
\end{split}
\end{equation}
for a.e. $\omega \in \Omega$. 

Let $E_J = \{ \omega \in \Omega \mid \widehat{I}(\omega) = J \}$ for any nonempty subset $J \subseteq I \cup \{ 0 \}$.
The sets $E_J$ form a partition of $\Omega$. Moreover, these sets are measurable, since the multifunction
$\widehat{I}(\cdot)$ is measurable. 

Observe that from \eqref{eq:OptCondDualForm} it follows that
\begin{align*}
  (\zeta(\omega), 0) &\in \underline{\partial}_{x, y} f(x_*, y_*(\omega), \omega) + (w_x(\omega), w_y(\omega))
  \\
  &+ c_* \co\Big\{ \underline{\partial}_{x, y} g_i(x_*, y_*(\omega), \omega) + (w_{xi}(\omega), w_{yi}(\omega))
  \Bigm| i \in J \Big\}
\end{align*}
for any $\omega \in E_J$ and any nonempty $J \subseteq I \cup \{ 0 \}$. With the use of the Filippov theorem 
(see, e.g. \cite[Thrm.~8.2.10]{AubinFrankowska}) one can readily check that the previous inclusion implies that for any
nonempty $J \subseteq I \cup \{ 0 \}$ there exist nonnegative measurable functions $\alpha_i^J(\cdot)$, $i \in J$, such
that $\sum_{i \in J} \alpha_i^J(\omega) = 1$ and
\begin{align*}
  (\zeta(\omega), 0) &\in \underline{\partial}_{x, y} f(x_*, y_*(\omega), \omega) + (w_x(\omega), w_y(\omega))
  \\
  &+ c_* \sum_{i \in J} \alpha_i(\omega) 
  \Big( \underline{\partial}_{x, y} g_i(x_*, y_*(\omega), \omega) + (w_{xi}(\omega), w_{yi}(\omega)) \Big)
\end{align*}
for a.e. $\omega \in E_J$. For any $i \in I$ define
\[
  \lambda_i(\omega) = \begin{cases}
    c_* \alpha_i^J(\omega), & \text{if } \omega \in E_J, \: i \in J 
  \text{ (or, equivalently, } i \in \widehat{I}(\omega)),
    \\
    0, & \text{ otherwise.}
  \end{cases}
\]	
Observe that by definition $\lambda_i$, $i \in I$, are nonnegative measurable functions such that 
$\sum_{i \in I} \| \lambda_i \|_{\infty} \le c_*$, and $\lambda_i(\omega) g_i(x_*, y_*(\omega), \omega) = 0$ for a.e.
$\omega \in \Omega$, since $\lambda_i(\omega) = 0$ whenever $i \notin \widehat{I}(\omega)$, i.e. 
$g_i(x_*, y_*(\omega), \omega) < 0$. Furthermore, bearing in mind the fact that $w_{x0}(\cdot) \equiv 0$, 
$w_{y0}(\cdot) \equiv 0$, and $\underline{\partial}_{x, y} g_0(x_*, y_*(\omega), \omega) \equiv \{ 0 \}$ one gets that
\begin{align*}
  (\zeta(\omega), 0) &\in \underline{\partial}_{x, y} f(x_*, y_*(\omega), \omega) + (w_x(\omega), w_y(\omega))
  \\
  &+ \sum_{i \in I} \lambda_i(\omega) 
  \Big( \underline{\partial}_{x, y} g_i(x_*, y_*(\omega), \omega) + (w_{xi}(\omega), w_{yi}(\omega)) \Big).
\end{align*}
for a.e. $\omega \in \Omega$. Hence applying equality \eqref{eq:QuasidiffViaCodiff} we arrive at the required result.
\end{proof}

\begin{remark}
It should be noted that with the use of the codifferential calculus one can compute a codifferential of the function
$f_0$ from the proof of the previous theorem, apply necessary conditions for a minimum of a codifferentiable function on
a convex set \cite[Thrm.~2.8]{DolgopolikMV_ESAIM} to the functional $\mathcal{J}$, and then directly rewrite these
conditions in terms of the problem $(\mathcal{P})$ with the use of Theorem~\ref{thrm:Codifferentiability} and 
an explicit expression for a codifferential of $f_0$. However, one can check that this approach leads to more cumbersome
optimality conditions than the ones from the theorem above. It is possible to verify that these conditions are
equivalent, but in the author's opinion the proof of this equivalence is more difficult than the proof of the previous
theorem. That is why we chose to present a simpler, but somewhat indirect derivation of optimality conditions for 
the problem $(\mathcal{P})$.
\end{remark}

\begin{remark}
Note that in the case when the functions $f$ and $g_i$, $i \in I$, are differentiable jointly in $x$ and $y$, the
optimality conditions from Theorem~\ref{thrm:OptCond} take the following well-known form 
(cf. \cite{RockafellarWets75,HiriartUrruty,XuYe10,ShapiroDentchevaRuszczynski,Vogel}). There exist nonnegative 
multipliers $\lambda_i \in \xL^{\infty}(\Omega, \mathfrak{A}, P)$, $i \in I$, such that 
$\lambda_i(\omega) g_i(x_*, y_*(\omega), \omega) = 0$ for a.e. $\omega \in \Omega$ and all $i \in I$, and
\begin{gather*}
  \left\langle \mathbb{E}\Big[ \nabla_x f(x_*, y_*(\cdot), \cdot) 
  + \sum_{i \in I} \lambda_i(\cdot) \nabla_x g_i(x_*, y_*(\cdot), \cdot) \Big], x - x_* \right\rangle \ge 0
  \quad \forall x \in A,
  \\
  \nabla_y f(x_*, y_*(\omega), \omega) + \sum_{i \in I} \lambda_i(\omega) \nabla_y g_i(x_*, y_*(\omega), \omega) = 0
  \quad \text{for a.e. } \omega \in \Omega.
\end{gather*}
\end{remark}

\section{Conclusions}

This work was devoted to an analysis of nonsmooth two-stage stochastic programming problems with the use of tools of
constructive nonsmooth analysis \cite{DemyanovRubinov}. In the first part of the paper, we analysed
the co-/quasi-differentiability of the expectation of nonsmooth random integrands and obtained explicit formulae for
its co-/quasi-differentials under some natural measurability and growth conditions on the integrand and its
codifferential.

In the second part of the paper, we obtained two types of sufficient conditions for the global exactness of a penalty
function for two-stage stochastic programming problems, reformulated as equivalent variational problems with pointwise
constraints. The first type of sufficient conditions is formulated for the penalty term defined via the $\xL^p$ norm of
the distance to the feasible set of the second stage problem, while the second type of sufficient conditions is
formulated for the penalty term that is independent of $p$ and is defined via the constraints of the second stage
problems. Although the second type of sufficient conditions is much more restrictive than the first one, it is more
convenient for applications and derivation of optimality conditions. Furthermore, as is pointed out in
Remark~\ref{rmrk:ConnectionToDCProblems}, these conditions open a way for the derivation of global optimality conditions
and application of DC optimization methods to two stage stochastic programming problems, whose second stage problem has
DC objective function and DC constraints.

Finally, in the last part of the paper we combined our results on codifferentiability of the expectation of nonsmooth
random integrands and exact penalty functions to derive optimality conditions for nonsmooth two-stage stochastic
programming problems in terms of codifferentials, involving essentially bounded Lagrange multipliers.

%%%
%	Bibliography
%%%

\bibliographystyle{abbrv}  
\bibliography{Dolgopolik_bibl}

\begin{thebibliography}{10}

\bibitem{AubinFrankowska}
J.-P. Aubin and H.~Frankowska.
\newblock {\em Set-Valued Analysis}.
\newblock Birkh\"{a}user, Boston, 1990.

\bibitem{BarbarosogluArda}
G.~Barbaroso\v{g}lu and Y.~Arda.
\newblock A two-stage stochastic programming framework for transportation
  planning in disaster response.
\newblock {\em J. Oper. Res. Soc.}, 55:43--53, 2004.

\bibitem{BirgeLouveaux}
J.~R. Birge and F.~Louveaux.
\newblock {\em Introduction to Stochastic Programming}.
\newblock Springer, New York, 2011.

\bibitem{Bogachev}
V.~I. Bogachev.
\newblock {\em Measure Theory. Volume {I}}.
\newblock Springer-Verlag, Berlin, Heidelberg, 2007.

\bibitem{BurkeChenSun}
J.~V. Burke.
\newblock The subdifferential of measurable composite max integrands and
  smoothing approximation.
\newblock {\em Math. Program.}, 181:229--264, 2020.

\bibitem{ChenFukushima}
X.~Chen and M.~Fukushima.
\newblock Expected residual minimization method for stochastic linear
  complementarity problems.
\newblock {\em Math. Oper. Res.}, 30:916--638, 2005.

\bibitem{ChenWetsZhang}
X.~Chen, {\relax R. J.-B. Wets}, and Y.~Zhang.
\newblock Stochastic variational inequalities: residual minimization smoothing
  sample average approximations.
\newblock {\em SIAM J. Optim.}, 22:649--673, 2012.

\bibitem{DempeKalashnikov}
S.~Dempe, V.~Kalashnikov, G.~A. P\'{e}rez-{V}ald\'{e}s, and N.~Kalashnykova,
  editors.
\newblock {\em Bilevel Programming Problems. Theory, Algorithms and
  Applications to Energy Networks}.
\newblock Springer, Berlin, Heidelberg, 2015.

\bibitem{DempeZemkoho}
S.~Dempe and A.~Zemkoho, editors.
\newblock {\em Bilevel Optimization. Advances and Next Challenges}.
\newblock Springer, Cham, 2020.

\bibitem{Demyanov2000}
V.~F. Demyanov.
\newblock Conditions for an extremum in metric spaces.
\newblock {\em J. Glob. Optim.}, 17:55--63, 2000.

\bibitem{Demyanov2010}
V.~F. Demyanov.
\newblock Nonsmooth optimization.
\newblock In G.~{D}i Pillo and F.~Schoen, editors, {\em Nonlinear optimization.
  Lecture notes in mathematics, vol. 1989}, pages 55--163. Springer-Verlag,
  Berlin, 2010.

\bibitem{DemyanovDixon}
V.~F. Demyanov and L.~C.~W. Dixon, editors.
\newblock {\em Quasidifferential Calculus}.
\newblock Springer, Berlin, Heidelberg, 1986.

\bibitem{DemyanovMalozemov}
V.~F. Dem'yanov and V.~N. Malozemov.
\newblock {\em Introduction to Minimax}.
\newblock Dover Publications, New York, 2014.

\bibitem{DemyanovRubinov}
V.~F. Demyanov and A.~M. Rubinov.
\newblock {\em Constructive Nonsmooth Analysis}.
\newblock Peter Lang, Frankfurt am Main, 1995.

\bibitem{DemyanovRubinov_collect}
V.~F. Demyanov and A.~M. Rubinov, editors.
\newblock {\em Quasidifferentiability and Related Topics}.
\newblock Kluwer Academic Publishers, Dordrecht, 2000.

\bibitem{Dolgopolik_CodiffCalc}
M.~V. Dolgopolik.
\newblock Codifferential calculus in normed spaces.
\newblock {\em J. Math. Sci.}, 173:441--462, 2011.

\bibitem{DolgopolikMV_ESAIM}
M.~V. Dolgopolik.
\newblock Nonsmooth problems of calculus of variations via codifferentiation.
\newblock {\em ESAIM: Control Optim. Calc. Var.}, 20:1153--1180, 2014.

\bibitem{Dolgopolik_AbstractCodiff}
M.~V. Dolgopolik.
\newblock Abstract convex approximations of nonsmooth functions.
\newblock {\em Optim.}, 64:1439--1469, 2015.

\bibitem{Dolgopolik_ExPenFunc}
M.~V. Dolgopolik.
\newblock A unifying theory of exactness of linear penalty functions.
\newblock {\em Optim.}, 65:1167--1202, 2016.

\bibitem{Dolgopolik_MCD}
M.~V. Dolgopolik.
\newblock A convergence analysis of the method of codifferential descent.
\newblock {\em Comput. Optim. Appl.}, 71:879--913, 2018.

\bibitem{Dolgopolik_arXiv}
M.~V. Dolgopolik.
\newblock Constrained nonsmooth problems of the calculus of variations.
\newblock {\em ESAIM: Control, Optim., Calc. Var.}, 27:1--35, 2021.

\bibitem{DolgopolikFominyh}
M.~V. Dolgopolik and A.~Fominyh.
\newblock Exact penalty functions for optimal control problems {I}: main
  theorem and free-endpoint problems.
\newblock {\em Optim. Control Appl. Meth.}, 40:1018--1044, 2019.

\bibitem{FabianSaoke}
C.~I. F\'{a}bi\'{a}n and Z.~Sz{\H{o}}ke.
\newblock Solving two-stage stochastic programming problems with level
  decomposition.
\newblock {\em Comput. Manag. Sci.}, 4:313--353, 2007.

\bibitem{FlamZowe}
S.~D. Fl{\aa}m and J.~Zowe.
\newblock Exact penalty functions in single-stage stochastic programming.
\newblock {\em Optim.}, 21:723--734, 1990.

\bibitem{GrassFischer}
E.~Grass and K.~Fischer.
\newblock Two-stage stochastic programming in disaster management: a literature
  survey.
\newblock {\em Surv. Oper. Res. Manag. Sci.}, 21:85--100, 2016.

\bibitem{HiriartUrruty}
J.~B. Hiriart-Urruty.
\newblock Conditions n\'{e}cessaires d'optimalit\'{e} pour un programme
  stochastique avec recours.
\newblock {\em SIAM J. Control Optim.}, 16:317--329, 1978.

\bibitem{HuangLoucks}
G.~H. Huang and D.~P. Loucks.
\newblock An inexact two-stage stochastic programming model for water resources
  management under uncertainty.
\newblock {\em Civ. Eng. Environ. Syst.}, 17:95--118, 2000.

\bibitem{LeoveyRomisch}
H.~Le\"{o}vey and W.~R\"{o}misch.
\newblock Quasi-{M}onte {C}arlo methods for linear two-stage stochastic
  programming problems.
\newblock {\em Math. Program.}, 151:315--345, 2015.

\bibitem{LinHuangXiaLi}
S.~Lin, M.~Huang, Z.~Xia, and D.~Li.
\newblock Quasidifferentiabilities of the expectation functions of random
  quasidifferentiable functions.
\newblock {\em Optim.}, pages 1--16, 2020.
\newblock DOI: 10.1080/02331934.2020.1818235.

\bibitem{LiuFanOrdonez}
C.~Liu, Y.~Fan, and F.~Ord\'{o}{\~n}ez.
\newblock A two-stage stochastic programming model for transportation network
  protection.
\newblock {\em Comput. Oper. Res.}, 36:1582--1590, 2009.

\bibitem{NemirovskiJuditskyLanShapiro}
A.~Nemirovski, A.~Juditsky, G.~Lan, and A.~Shapiro.
\newblock Robust stochastic approximation approach to stochastic programming.
\newblock {\em SIAM J. Optim.}, 19:1574--1609, 2009.

\bibitem{OliveiraSagastizabal}
W.~Oliveira, C.~Sagastiz\'{a}bal, and S.~Scheimberg.
\newblock Inexact bundle methods for two-stage stochastic programming.
\newblock {\em SIAM J. Optim.}, 21:517--544, 2011.

\bibitem{Orlov}
A.~V. Orlov.
\newblock On a solving bilevel {D}.{C}.-convex optimization problems.
\newblock In Y.~Kochetov, I.~Bukadorov, and T.~Gruzdeva, editors, {\em
  Mathematical Optimization Theory and Operations Research. MOTOR 2020.}, pages
  179--191. Springer, Cham, 2020.

\bibitem{DiPilloGrippo}
G.~\relax{Di Pillo} and L.~Grippo.
\newblock Exact penalty functions in constrained optimization.
\newblock {\em SIAM J. Control Optim.}, 27:1333--1360, 1989.

\bibitem{RockafellarWets}
R.~T. Rockafellar and \relax{R. J.-{B}. Wets}.
\newblock {\em Variational Analysis}.
\newblock Springer-Verlag, Berlin, Heidelberg, 1998.

\bibitem{RockafellarWets75}
R.~T. Rockafellar and {\relax R. J.-B. Wets}.
\newblock Stochastic convex programming: {K}uhn-{T}ucker conditions.
\newblock {\em J. Math. Econ.}, 2:349--370, 1975.

\bibitem{RockafellarWets82}
R.~T. Rockafellar and {\relax R. J.-B. Wets}.
\newblock On the interchange of subdifferentiation and conditional expectation
  for convex functions.
\newblock {\em Stochastics}, 7:173--182, 1982.

\bibitem{RubinovYang}
A.~Rubinov and X.~Yang.
\newblock {\em Lagrange-Type Functions in Constrained Non-Convex Optimization}.
\newblock Kluwer Academic Publishers, Boston, 2003.

\bibitem{ShapiroHomemDeMello}
A.~Shapiro and T.~H. de~{M}ello.
\newblock A simulation-based approach to two-stage stochastic programming with
  recourse.
\newblock {\em Math. Program.}, 81:301--325, 1998.

\bibitem{ShapiroDentchevaRuszczynski}
A.~Shapiro, D.~Dentcheva, and A.~Ruszcz\`{n}ski.
\newblock {\em Lectures on Stochastic Programming: Modeling and Theory}.
\newblock SIAM, Philadelphia, 2014.

\bibitem{Strekalovsky}
A.~S. Strekalovsky.
\newblock Global optimality conditions and exact penalization.
\newblock {\em Optim. Lett.}, 13:597--615, 2019.

\bibitem{StrekalovskyOrlov}
A.~S. Strekalovsky and A.~V. Orlov.
\newblock Global search for bilevel optimization with quadratic data.
\newblock In S.~Dempe and A.~Zemkoho, editors, {\em Bilevel Optimization},
  pages 313--334. Springer, Cham, 2020.

\bibitem{Uderzo}
A.~Uderzo.
\newblock On the variational behaviour of functions with positive steepest
  descent rate.
\newblock {\em Positivity}, 19:725--745, 2015.

\bibitem{Uderzo2}
A.~Uderzo.
\newblock A strong metric subregularity analysis of nonsmooth mappings via
  steepest displacement rate.
\newblock {\em J. Optim. Theory Appl.}, 171:573--599, 2016.

\bibitem{Vogel}
S.~Vogel.
\newblock Necessary optimality conditions for two-stage stochastic programming
  problems.
\newblock {\em Optim.}, 16:607--616, 1985.

\bibitem{XuYe10}
H.~Xu and J.~J. Ye.
\newblock Necessary optimality conditions for two-stage stochastic programs
  with equilibrium constraints.
\newblock {\em SIAM J. Optim.}, 20:1685--1715, 2010.

\bibitem{XuZhang09}
H.~Xu and D.~Zhang.
\newblock Smooth sample average approximation of stationary points in nonsmooth
  stochastic optimization and applications.
\newblock {\em Math. Program.}, 119:371--401, 2009.

\bibitem{Zaslavski}
A.~J. Zaslavski.
\newblock {\em Optimization on Metric and Normed Spaces}.
\newblock Springer, New York, 2010.

\bibitem{ZhouZhangEtAl}
Z.~Zhou, J.~Zhang, P.~Liu, Z.~Li, M.~C. Georgiadis, and E.~N. Pistikopoulos.
\newblock A two-stage stochastic programming model for the optimal design of
  distributed energy systems.
\newblock {\em Appl. Energy}, 103:135--144, 2013.

\end{thebibliography}

\end{document}